\documentclass[notitlepage,11pt,reqno]{amsart}

\usepackage{amsopn,esint}
\usepackage[final]{hyperref}
\usepackage[T1]{fontenc}

\usepackage[utf8]{inputenc}%for appendix
\usepackage{apptools}%for appendix
\AtAppendix{\counterwithin{lemma}{section}} %for appendix

\usepackage[margin=1in]{geometry}
\newcommand{\stkout}[1]{\ifmmode\text{\sout{\ensuremath{#1}}}\else\sout{#1}\fi}

 \newcommand{\grad}{\triangledown}

\newcommand{\sA}{\mathscr{A}}
\newcommand{\SB}{\mathcal{T}}
\usepackage{graphicx,enumitem,dsfont,upgreek}
 \usepackage[dvips]{epsfig}
 \usepackage[mathscr]{eucal}
\usepackage{amscd}
\usepackage{amssymb}
\usepackage{amsthm}
\usepackage{amsmath}
\usepackage{latexsym}
\usepackage{dsfont}
\usepackage{upref}
\usepackage{hyperref}

\usepackage{color}
\theoremstyle{plain}

\newtheorem{thm}{Theorem}[section]
\theoremstyle{plain}
\newtheorem{lem}[thm]{Lemma}

\newtheorem{example}{Example}[section]
\newtheorem{assumption}{Assumption}[section]
\theoremstyle{definition}
\newtheorem{defi}{Definition}[section]
\newtheorem{rem}{Remark}[section]
\newtheorem*{maintheorem*}{Main Theorem}
\newtheorem*{maincorollary*}{Main Corollary}
{%
\setcounter{enumi}{0}

\begin{enumerate}}%
{\end{enumerate} }

{%
\setcounter{enumi}{0}

\begin{enumerate}}%
{\end{enumerate} }

\newcommand{\norm}[1]{\ensuremath{\left\|#1\right\|}}
\newcommand{\abs}[1]{\ensuremath{\left|#1\right|}}

\newcommand{\Kset}{\ensuremath{\mathcal{K}}}
\newcommand{\sK}{\mathscr{K}}

\newcommand{\cL}{\ensuremath{\mathcal{L}}}
\newcommand{\cI}{\ensuremath{\mathcal{I}}}
\newcommand{\sV}{\ensuremath{\mathscr{V}}}

\newcommand{\R}{\ensuremath{\mathbb{R}}}

\newcommand{\sorder}{\mathfrak{o}}
\newcommand{\sL}{\mathscr{L}}
\newcommand{\Usm}{\mathscr{U}}
\newcommand{\sJ}{\mathscr{J}}
\newcommand{\rd}{\ensuremath{\R^d}}

\newcommand{\dy}{\ensuremath{\, dy}}
\newcommand{\dz}{\ensuremath{\, dz}}

\numberwithin{equation}{section} \allowdisplaybreaks

\DeclareMathOperator*{\Argmin}{arg\,min}
\usepackage{cancel,pdfsync}

\begin{document}
\title[Existence-Uniqueness results]{Existence-Uniqueness for nonlinear integro-differential
equations with drift in $\rd$}

\author{Anup Biswas 
\& Saibal Khan}

\address{Indian Institute of Science Education and Research-Pune, Dr.\ Homi Bhabha Road, Pashan, Pune 411008. Email:
{\tt anup@iiserpune.ac.in, saibal.khan@acads.iiserpune.ac.in}}

\dedicatory{Dedicated to the memory of Ari Arapostathis (1954-2021)}

\begin{abstract}
In this article we consider a class of
 nonlinear integro-differential 
equations of the form
$$\inf_{\tau \in\mathcal{T}} \bigg\{\int_{\mathbb{R}^d}
(u(x+y)+u(x-y)-2u(x))\frac{k_{\tau}(x,y)}{|y|^{d+2s}} \,dy+ b_{\tau}(x) \cdot \nabla u(x)+g_{\tau}(x) \bigg\}-\lambda^*=0\quad \text{in} \hspace{2mm} \mathbb{R}^d,$$
where $0<\lambda(2-2s)\leq k_{\tau}\leq \Lambda (2-2s)$ , $s\in (\frac{1}{2},1)$. The above equation appears in the study of ergodic control problems in $\rd$ when the controlled dynamics is governed by pure-jump L\'evy processes characterized by the kernels $k_{\tau}\,|y|^{-d-2s}$ and the drift $b_\tau$.
Under a Foster-Lyapunov condition, we establish the existence of a unique solution pair $(u, \lambda^*)$ satisfying the above equation, provided we set $u(0)=0$.
Results are then extended to cover the HJB equations 
of mixed local-nonlocal type and this significantly improves the results
in \cite{ACGZ}.
\end{abstract}


\maketitle

\section{Introduction}
Our chief goal in this article is to find a pair 
  $(u,\lambda^*)$ that satisfies  
\begin{align}
\label{ergodic_HJB}
\inf_{\tau \in\mathcal{T}} \bigg\{\int_{\mathbb{R}^d}\delta (u,x,y)\frac{k_{\tau}(x,y)}{|y|^{d+2s}} \,dy + b_{\tau}(x) \cdot \nabla u(x)+g_{\tau}(x) \bigg\}-\lambda^*=0\quad \text{in} \hspace{2mm} \mathbb{R}^d,
\end{align}
  where $\delta(u,x,y):=u(x+y)+u(x-y)-2 u(x)$ and $\mathcal{T}$ is an indexing set.  We impose the following assumptions on the kernel: Let $x\mapsto  k_{\tau}(x,y)$ be continuous uniformly in $(y,\tau)$ and satisfies 
 \begin{align*}
 k_\tau(x,y)=k_\tau(x,-y),\quad
(2-2s)\lambda \leq k_{\tau}(x,y)\leq  (2-2s)\Lambda \quad \forall\, x, y \in \mathbb{R}^d,
 \end{align*}
 where  $s\in (\frac{1}{2},1)$ and $0< \lambda \leq \Lambda$.
Let us introduce  the following notations.
\begin{align}\label{lin-opt}
I_{\tau}[u](x)=\int_{\mathbb{R}^d}\delta (u,x,y)\frac{k_{\tau}(x,y)}{|y|^{d+2s}} \,dy;\quad \mathcal{L}_{\tau}[u](x)=I_{\tau}[u](x)+ b_{\tau}(x) \cdot \nabla u(x).
\end{align}
Also, denote by $\omega_s(y)=\frac{1}{1+|y|^{d+2s}}$. 
\begin{rem}
The symmetry property of the kernel $k_\tau$ is not used in this article, but we still make this assumption due to the following reason. In general, the nonlocal operator of 
L\'evy type does not have a symmetrized form (cf. \cite[Chapter~3]{Apple}) and this symmetrization (that is, the form involving $\delta(u,x,y)$) is possible when kernel $k_\tau$ is symmetric in the second variable. Therefore, keeping the
assumption of symmetry makes the model physically relevant.
\end{rem}

One of the main motivations to study \eqref{ergodic_HJB} comes from the stochastic ergodic control problems where the
random noise in the controlled dynamics corresponds to some $2s$-stable process. More precisely, suppose that 
the action set (or control set) $\mathcal{T}$ is a metric space and 
$\Usm$ denotes the collection of all stationary Markov controls, that is, collection of all Borel measurable functions $v:\mathbb{R}^d\to \mathcal{T}$. This class of functions plays a central role in the study of optimal control problems. Let us also assume that the martingale problem corresponding to the operator $\mathcal{L}^v$ , defined by (see \eqref{lin-opt})
$$\mathcal{L}^v [f](x)=\mathcal{L}_{v(x)}[f](x),$$
is well-posed. In particular, for every $v\in\Usm$ there exists a family of probability measures $\{\mathbb{P}_x\}_{x\in\mathbb{R}^d}$ on $\mathbb{D}([0, \infty), \mathbb{R}^d)$, the space of C\`{a}dl\`{a}g functions on $[0, \infty)$ taking values in $\mathbb{R}^d$, such that $(\mathbb{P}^v_x, X^v)$, where $X^v$ denotes the canonical coordinate process,
solves the martingale problem. One albeit needs to impose certain regularity hypothesis on the kernels $k_\tau$ to guarantee well-posedness of the martingale problem, see for instance \cite{CKS12,Kom84}. Let $\rd\times\mathcal{T}\ni (x, \tau)\to
g_\tau(x)$ denotes the running cost
 and the goal is to minimize the 
ergodic cost criterion
$$\sJ[v]:=\limsup_{T\to\infty}\, \frac{1}{T} \mathbb{E}^v_x\left[\int_0^T g_{v}(X^v_t)\right],$$
over $\Usm$. We denote the optimal value by $\lambda^*$.
It is then expected that the optimal value $\lambda^*$ would satisfy \eqref{ergodic_HJB} (cf. \cite{ACGZ,red-book,Bor89,GM92,FS06})
and the measurable selectors of \eqref{ergodic_HJB} would be the optimal controls in $\Usm$. Though the analogous problem 
for the local case (that is, $s=1$) has been investigated extensively (see \cite{red-book} and references therein), study of equation \eqref{ergodic_HJB} remained open. 
Recently, ergodic control problem in $\rd$ with dispersal type nonlocal kernel is considered by Br\"{a}ndle \& Chasseigne in \cite{BC19}
whereas Barles et.\ al. \cite{BCCI14} study the ergodic control problem for mixed integro-differential operators in periodic settings.
In this article, we establish the existence and uniqueness of solution to \eqref{ergodic_HJB} under a Foster-Lyapunov type 
condition.

We say
a function $f:\rd\to \R$ is inf-compact (or coercive) if for any
$\kappa\in\R$ either $\{f\leq \kappa\}$ is empty or compact.
\begin{assumption}
We make following assumptions on the coefficients.
\begin{itemize}\label{assumptions}
\item [(\hypertarget{A1}{A1})] There exists $V\in C^2(\mathbb{R}^d), V\geq 0$ and a function $0\leq h\in C(\mathbb{R}^d)$, $V$ and $h$ are inf-compact, such that
\begin{align}\label{Lyap}
\sup_{\tau \in\mathcal{T}}\mathcal{L}_{\tau} [V](x)\leq k_0-h(x)
\quad x\in\rd,
\end{align}
for some $k_0>0$. This is called Foster-Lyapunov stability condition. The above condition also implies that $V\in L^1(\omega_s)$.
\item[(\hypertarget{A2}{A2})] $\sup_{\tau\in\mathcal{T}}|g_{\tau}(x)|\leq h(x)$  and 
$\sup_{\tau \in\mathcal{T}}|g_{\tau}|\in\sorder(h)$, that is,
$$\lim_{|x|\to\infty} \frac{1}{1+h(x)}\sup_{\tau \in\mathcal{T}}|g_\tau(x)|=0.$$
\item[(\hypertarget{A3}{A3})] For some $\upmu\geq 0$ we have $V^{1+\upmu}\in L^1(\omega_s)$ and
\begin{align}
\sup_{\tau \in\mathcal{T}}\left[\frac{|b_\tau|}{(1+V)^{(2s-1)\upmu}}
+ \frac{|g_{\tau}|}{(1+V)^{1+2s\upmu}}\right]&\leq C,\label{EA.3A}
\\ 
\sup_{x, y}\frac{V(x+y)}{(1+V(x))(1+V(y))} + \limsup_{|x|\to\infty}\; \frac{1}{1+V(x)}\sup_{|y-x|\leq 1} V(y)& \leq C.\label{EA.3B}
\end{align}
\item[(\hypertarget{A4}{A4})] The map $x\rightarrow k_{\tau}(x,y)$ is uniformly continuous, uniformly in $\tau$ and $y$, that is,
  \begin{align*}
  |k_{\tau}(x_1,y)-k_{\tau}(x_2,y)|\leq \varrho(|x_1-x_2|)\quad \forall x_1,x_2 \in \mathbb{R}^d,\; \tau\in \mathcal{T}, y\in\rd,
  \end{align*}
for some modulus of continuity $\varrho$.
\end{itemize}
\end{assumption}
Note that \eqref{EA.3A} allows $b_\tau, g_{\tau}$ to be unbounded. Since
$V$ is inf-compact, \eqref{EA.3A} holds for bounded $b_\tau, g_{\tau}$. This condition will be useful to find a growth-bound
on the H\"{o}lder norm of solution $u$ in the unit ball
$B_1(x)$.
Condition \eqref{EA.3B} requires $V$ to have polynomial growth. For $C^{2s+\upkappa}$ regularity of solutions we shall often impose the following condition on the coefficients.
\begin{itemize}
\item[(\hypertarget{A5}{A5})] For every compact set $K\subset\rd$, there exists $\tilde\alpha\in (0, 1)$ so that
$$\sup_{\tau\in \mathcal{T}}\abs{b_\tau(x)-b_\tau(x')} 
+ \sup_{\tau\in \mathcal{T}}\abs{g_{\tau}(x)-g_{\tau}(x')}
+ \sup_{\tau \in\mathcal{T}}\,\sup_{y\in\rd}
|k_{\tau}(x, y)-k_{\tau}(x', y)|\leq C_K|x-x'|^{\tilde\alpha}.$$
\end{itemize}

It can be easily seen that Assumption~\ref{assumptions} holds for $2s$-fractional Ornstein-Uhlenbeck type operators. Such operators are used by Fujita, Ishii \& Loreti \cite{FIL06} to study large time behaviour of solutions to certain (local) Hamilton-Jacobi equation. Recently, Chasseigne, Ley \& Nguyen \cite{CLN19} consider Ornstein-Uhlenbeck type drift with tempered $2s$-stable nonlocal kernel to study Lipschitz regularity of the
solutions.
In Example~\ref{Eg1.1}
below we provide a large class of functions satisfying
(\hyperlink{A1}{A1})--(\hyperlink{A3}{A3}). 

One of the main results of this article is as follows
\begin{thm}\label{T1.1}
Suppose that (\hyperlink{A1}{A1})--(\hyperlink{A5}{A5}) hold. Then there exists a
unique pair $(u, \lambda^*)\in
C(\rd)\cap \sorder(V)\times\mathbb{R}$ satisfying
\begin{equation}\label{ET1.1A}
\inf_{\tau \in\mathcal{T}}(\cL_{\tau} u + g_{\tau}) -\lambda^*=0\quad \text{in}\; \rd, \quad u(0)=0.
\end{equation}
\end{thm}
By a solution we  shall always mean viscosity solution in the sense of Caffarelli-Silvestre \cite{CS09} (see also \cite{BI08,BCI08}), unless mentioned otherwise.
\begin{defi}[Viscosity solution]
A function $u:\rd\to \R$, upper (lower) semi continuous in a domain $\bar{D}$ and $u\in L^1(\omega_s)$, is said to be a viscosity subsolution (supersolution) to
$$\inf_{\tau \in\mathcal{T}}(\cL_{\tau} u + g_{\tau})=0\quad \text{in}\; D,$$
written as $\inf_{\tau \in\mathcal{T}}(\cL_{\tau} u + g_{\tau})\geq 0$ ($\inf_{\tau \in\mathcal{T}}(\cL_{\tau} u + g_{\tau})\leq 0$), if for any point $x\in D$ and a neighbourhood
$N_x$ of $x$ in $D$, there exists a function $\varphi\in C^2(\bar{N}_x)$ so that $\varphi-u$ attains minimum value $0$ in $N_x$ at the point $x$, then letting
\[
v(y):=\left\{
\begin{array}{ll}
\varphi(y) & \text{for}\; y\in N_x,
\\[2mm]
u(y) & \text{otherwise},
\end{array}
\right. 
\]
we have $\inf_{\tau \in\mathcal{T}}(\cL_{\tau} v + g_{\tau})\geq 0$ ($\inf_{\tau \in\mathcal{T}}(\cL_{\tau} v + g_{\tau})\leq 0$, resp.). We say $u$ is a viscosity solution if
it is both sub and supersolution.
\end{defi}
 
In practice, there are two types of situations that are normally considered to study ergodic control problems. These
are (i) near-monotone condition, and (ii) stability condition
(cf. \cite{ACGZ,red-book}). Under near-monotone setting, it is assumed that the limiting value of 
$\inf_{\tau\in\mathcal{T}}g_\tau$ at infinity is strictly bigger than $\lambda^*$ (cf. (2.3) in \cite{ACGZ}).
 The classical linear-quadratic problems fall under this setting. Due to this near-monotone condition we expect the optimal Markov control, if exists, to be {\it stable}.
In the second situation ( that is, under stability condition), a blanket stability condition (same as \eqref{Lyap}) is imposed on the drifts but no growth condition at infinity is imposed on the running cost $g_\tau$. Though
\eqref{ET1.1A} has been extensively studied in the local case, the nonlocal version remained unsettled, mainly due to the unavailability of certain regularity estimates and appropriate Harnack type inequality for general nonlinear integro-differential operators. 
To understand the difficulty in proving
Theorem~\ref{T1.1} we recall the three main steps of the proof: 
\begin{itemize}
\item[(a)] Find solution $w_\alpha$ to the $\alpha$-discounted problem
for $\alpha\in (0,1)$; 
\item[(b)] Establish the convergence of $\bar{w}_\alpha(x):= w_\alpha(x)-w_\alpha(0)$ to solution $u$ of \eqref{ET1.1A};
\item[(c)] Show that $(u, \lambda^*)$ is unique.
\end{itemize}
For (a), we consider a more general class of discounted problem given by
$$\cI w(x)=\inf_{\tau \in\mathcal{T}}\bigl(\cL_{\tau} w + c_{\tau}(x) w(x)+ g_{\tau}(x)\bigr)=0\quad \text{in}\; \rd.
$$
Fixing $c_\tau=-\alpha$ we obtain $\alpha$-discounted problem. It is possible to weaken Assumption~\ref{assumptions} substantially to solve the above
equation  (see Assumption~\ref{A2.1} in Section~\ref{S-alpha}).
In particular, we prove the following result.
\begin{thm}\label{T1.2}
Suppose that (\hyperlink{B1}{B1})--(\hyperlink{B4}{B4}) hold.
Then there exists $w\in C(\rd)\cap\sorder(\sV)$,
$\sV$ given by \eqref{Lyap-alpha}, satisfying
\begin{equation}\label{ET1.2A}
\inf_{\tau \in\mathcal{T}} \Bigl(\cL_{\tau} w + c_{\tau} w + g_{\tau}\Bigr) =0\quad \text{in}\; \rd.
\end{equation}
In addition, if (\hyperlink{B5}{B5}) holds, then 
$w\in C^{2s+}_{\rm loc}(\rd)$ and it is the unique solution
in the class $\sorder(\sV)$.
\end{thm}
If compared with the existing literature, the existence, uniqueness and the regularity results in Theorem~\ref{T1.2} appear to be new. For instance, most of the existence results in
the unbounded domains consider periodic setting, see Barles et.\ al. \cite{BCCI14}, Ciomaga, Ghilli \& Topp \cite{CGT22}.  Under the 
assumption that the coefficients are Lipschitz with sublinear growth, Jakobsen \& Karlsen \cite{JK06} establish comparison principle for bounded viscosity solutions to mixed L\'evy-It\^{o} type Isaacs equations.
Biswas, Jakobsen \& Karlsen \cite{BJK10} study the existence-uniqueness of viscosity
solutions for parabolic integro-differential equations of L\'evy-It\^{o} type under the assumption that the L\'evy kernel is having an exponentially decaying tail and the coefficients of \eqref{ET1.2A} are Lipschitz.
The existence-uniqueness result in Theorem~\ref{T1.2} is obtained for a purely nonlocal equation that is not of L\'evy-It\^{o} type. Similar existence-uniqueness result for mixed integro-differential operator is 
established in Section~\ref{S-mixed}. In a recent work, Meglioli \& Punzo \cite{MP22} consider linear equation involving $2s$-fractional Laplacian and $C^1$ drift and study uniqueness of the solutions. The methodology
in \cite{MP22} is of variational nature and can not be adapted for operators of the form \eqref{ET1.2A} (see Remark~\ref{R2.1} for a detailed comparison). The regularity result plays a crucial role in obtaining the uniqueness. Two main ingredients to find  $C^{2s+}$
regularity are $C^{1, \gamma}$ regularity of $w$ and the growth bounded on the H\"{o}lder norm of $w$ in $B_1(x)$ (see Lemma~\ref{L2.6}). We establish the $C^{1, \gamma}$ regularity for Isaacs type equations
and without requiring the coefficients to be H\"{o}lder continuous (see Theorem~\ref{c_1,gamma_regularity} below for more detail). There is a large body of works dealing with the H\"{o}lder regularity of Isaacs type
integro-differential equations. Many of these works are based on Ishii-Lions method and therefore, they require the coefficients to be H\"{o}lder continuous and $c_\tau$ to be negative, see for
instance, \cite[Corollary~4.1]{BCI11}, \cite{BCCI12,CLN19}, \cite[Theorem~3.1]{CGT22}. We do not require any such condition since our method is based on scaling argument and Liouville theorem used by
Serra in  \cite{Serra2015,Serra_Parabolic}.

Since $c_\tau$ is negative (by (\hyperlink{B1}{B1})) we can 
solve \eqref{ET1.2A} in bounded domains, for instance, ball 
$B_n(0)$, with Dirichlet exterior condition. Then, applying the local H\"{o}lder estimate from  \cite{Schwab_Silvestre} one can pass to the limit, as $n\to\infty$,
to obtain a solution of \eqref{ET1.2A}, provided the solutions in the bounded domains are dominated by a fixed barrier function. This is required to pass the limit inside the nonlocal operator. We show that $\sV$ (see \eqref{Lyap-alpha}) can be used as a barrier function. To establish uniqueness, we first show that 
$w\in C^{2s+}_{\rm loc}(\rd)$ and therefore, it is a classical solution to \eqref{ET1.2A}. Recall that $C^{2s+}$ estimate of $w$ requires global H\"{o}lder regularity of 
$w$  and local H\"{o}lder regularity of $\grad w$ (cf. \cite{Serra2015}). To attain this goal we first establish 
$C^{1, \gamma}$ estimate for $w$.
This is done in Section~\ref{S-regu}. Once we have $C^{2s+}$
regularity (see Lemma~\ref{L2.6}) we can couple two solutions and use the barrier function $\sV$ to establish uniqueness.

Coming back to Theorem~\ref{T1.1}, we apply Theorem~\ref{T1.2} to obtain solution $w_\alpha$ for the $\alpha$-discounted problem and define the
normalized function $\bar{w}_\alpha(x)=w_\alpha(x)-w_\alpha(0)$. Note that
$$
\inf_{\tau \in\mathcal{T}}(\cL_{\tau} \bar{w}_\alpha + g_{\tau}) -\alpha \bar{w_\alpha} -\alpha w_\alpha(0)=0\quad \text{in}\; \rd, \quad \bar{w}_\alpha(0)=0.
$$
Thus to complete step (b) we only need to find a convergent 
subsequence of $\{\bar{w}_\alpha\}$, as $\alpha\to 0$, and pass the limit in the above equation. The equicontinuity of
the family $\{\bar{w}_\alpha\}$ is generally obtained
by employing a generalized Harnack's type estimate (cf. \cite[Theorem~3.3]{ACGZ},\cite[Lemma~3.6.3]{red-book}). 
But, to the best of our knowledge, 
Harnack type estimate is not available for nonlinear integro-differential operators with gradient and zeroth order term.
Therefore, we innovate a different method. We again use the Lyapunov function $V$ in \eqref{Lyap} and show that
$|\bar{w}_\alpha|\leq \kappa + V$ in $\rd$ for some 
suitable constant $\kappa$ (see Lemma~\ref{L3.1}). Now applying the regularity result of \cite{Schwab_Silvestre} we can establish the equicontinuity of $\{\bar{w}_\alpha\}$.
For the step (c), we again show that $u\in C^{2s+}_{\rm loc}(\rd)$ and then use the fact $u\in\sorder(V)$. It is worth pointing out that
equation \eqref{ET1.1A} is not strictly monotone which makes the uniqueness tricky. There are only few works dealing with the
uniqueness of non-monotone operators in bounded domains, see 
Caffarelli \& Silvestre \cite{CS09}, Mou \& \'{S}wi\k{e}ch \cite{Mou-Swiech}. The $C^{2s+}$ regularity property is crucially used in the proof of uniqueness. Note that once we have $C^{1, \gamma}$ regularity, $C^{2s+}$ estimate follows from Serra \cite{Serra2015} since \eqref{ET1.1A} is of concave type. Similar approach does not work for
Isaacs type equations.

Interestingly, the above approach can be generalized to study HJB equations
for a larger family of integro-differential operators. In particular, we consider the operator
$$\cI u(x) =\inf_{\tau\in\mathcal{T}}\left[
{\rm tr}{(a_\tau(x)D^2u)} + \breve I_\tau [u](x) + b_\tau(x)\cdot\grad u + g_\tau \right],$$
where $\breve I_\tau$ is a general L\'evy type nonlocal operator (see Section~\ref{S-mixed}). We show that the above approach extends for these class of operators and we obtain the following result in Section~\ref{S-mixed}.
\begin{thm}\label{T1.3}
Suppose that Assumptions~\ref{A4.1} and ~\ref{A4.2}  hold. Then there exists a
unique pair $(u, \lambda^*)\in\sorder(V)\times\mathbb{R}$ satisfying
\begin{equation}\label{ET4.1A}
\cI u(x) -\lambda^*=0\quad \text{in}\; \rd, \quad u(0)=0.
\end{equation}
\end{thm}
Theorem~\ref{T1.3} should be compared with Arapostathis et.\ al. \cite{ACGZ} where similar problem has been considered for nonlocal kernels having compact support and finite measure.
We conclude the introduction with an example satisfying Assumption~\ref{assumptions}.
This example is inspired from \cite{ABC16}.
\begin{example}\label{Eg1.1}
Consider a function $0\leq V\in C^2(\mathbb{R}^d)$ satisfying $V(x)=|x|^\upgamma$ for $|x|\geq 1$, and $\upgamma\in (0,2s)$. 
We observe that for any $|x|\leq 1$, $|V(x)|\leq C$  and therefore $V(x)\leq C+|x|^\upgamma$ for all $x \in \mathbb{R}^d$. All the inequalities below are true upto a constant.

\noindent\textbf{Case-I} Let $x\in B_R$ and $R>2$, then
\begin{align*}
\bigg|\int_{B_{R+1}}\delta(V,x,y)\frac{k_{\tau}(x,y)}{|y|^{d+2s}} \,dy\bigg|\leq ||D^2V||_{L^\infty(B_{2R+1})} \int_{{B_{R+1}}}\frac{|y|^2}{|y|^{d+2s}}\,dy =\norm{D^2V}_{L^\infty(B_{2R+1})}  (R+1)^{(2-2s)}
\end{align*}
Now observe that for $|y|\geq R+1$ and $|x|\leq R$, we have $|x\pm y|\geq 1$. Therefore 
\begin{align*}
\bigg|\int_{B^c_{R+1}}\delta(V,x,y)\frac{k_{\tau}(x,y)}{|y|^{d+2s}} \,dy\bigg|=&\bigg|\int_{B^c_{R+1}}\frac{|x+y|^\upgamma+|x-y|^\upgamma-2V(x)}{|y|^{d+2s}} \,dy\bigg| \\
\lesssim &\int_{B^c_{R+1}}\frac{2|y|^\upgamma +2(|x|^\upgamma+|V(x)|)}{|y|^{d+2s}} \,dy
\\
\lesssim & (R+1)^{\upgamma -2s} + (R^\upgamma +||V||_{L^\infty(B_{R})})R^{-2s}.
\end{align*}

\noindent\textbf{Case-2} Let $x \in B^c_R$. Then
\begin{align}\label{E1.1A}
\bigg|\int_{B_{\frac{R}{2}}}\delta(V,x,y)\frac{k_{\tau}(x,y)}{|y|^{d+2s}} \,dy\bigg|\leq ||D^2V||_{L^\infty(\mathbb{R}^d)}C_R.
\end{align}
Now to compute the integration on $B^c_{\frac{R}{2}}$ we define
\begin{align*}
J_{\tau}:=\int_{\frac{R}{2}\leq |y|\leq \frac{|x|}{2}}\delta(V,x,y) \frac{k_{\tau}(x,y)}{|y|^{d+2s}} \,dy .
\end{align*}
Note that $|y|\leq \frac{|x|}{2}\implies |x\pm y|\geq |x|-\frac{|x|}{2}\geq \frac{|x|}{2}\geq \frac{R}{2} \geq 1$. Therefore,
\begin{align*}
J_{\tau}&=\int_{\frac{R}{2}\leq |y|\leq \frac{|x|}{2}}(|x+ y|^\upgamma+|x-y|^\upgamma -2|x|^\upgamma) \frac{k_{\tau}(x,y)}{|y|^{d+2s}} \,dy
\\
&= |x|^{\upgamma-2s}\int_{ \frac{R}{2|x|}\leq |z|<\frac{1}{2}}\bigg(\bigg|\frac{x}{|x|}+ z\bigg|^\upgamma+\bigg|\frac{x}{|x|}-z\bigg|^\upgamma -2\bigg) \frac{k_{\tau}(x,|x|z)}{|z|^{d+2s}} \,dz
\end{align*}
Now observe that for any $|z|<\frac{1}{2}$ we have  $\big|\frac{x}{|x|}\pm z\big|\geq 1 -|z|\geq \frac{1}{2}$. Therefore we can apply Taylor's formula to conclude
\begin{align}\label{E1.1B}
J_{\tau}\leq \tilde{C}|x|^{\upgamma-2s} \int_{  |z|<\frac{1}{2}}\frac{|z|^2}{|z|^{d+2s}}\,dz\leq C |x|^{\upgamma-2s}.
\end{align}
Let us denote by $\tilde{\delta}(x,y)=|x+y|^\gamma+|x-y|^\gamma-2|x|^\gamma$ and observe that
\begin{align}\label{E1.1C}
\int_{\frac{|x|}{2}\leq |y|}\delta(V,x,y)\frac{k_{\tau}(x,y)}{|y|^{d+2s}}\,dy=\int_{\frac{|x|}{2}\leq |y|}\Big(\delta(V,x,y)-\tilde{\delta}(x,y)\Big)
\frac{k_{\tau}(x,y)}{|y|^{d+2s}}\,dy+\int_{\frac{|x|}{2}\leq |y|}\tilde{\delta}(x,y)\frac{k_{\tau}(x,y)}{|y|^{d+2s}}\,dy.
\end{align}
Furthermore, as for $|x|<2 |y|$ we have $||x\pm y|^\gamma-|x|^\gamma|\leq 9 |y|^\gamma$, it can be easily seen that
\begin{align*}
\bigg|\int_{\frac{|x|}{2}\leq |y|}\tilde{\delta}(x,y)K_{\tau}(x,y)\,dy\bigg|\leq C \int_{\frac{|x|}{2}\leq |y|} \frac{|y|^\upgamma}{|y|^{d+2s}}\,dy\leq C|x|^{\upgamma-2s}.
\end{align*}
Now to calculate the first integral
on the rhs of \eqref{E1.1C}, we observe that if both $|x+y|$ and $|x-y|$ are bigger than $1$, then $\delta(V,x,y)-\tilde{\delta}(x,y)=0$. Again, if $|x + y|\leq 1$ (or $|x - y|\leq 1$) then
\begin{align*}
\frac{|x|}{2}\leq |y|\leq |x|+|x+y|\leq |x|+1\leq \frac{3|x|}{2},
\end{align*}
and hence,
\begin{align}\label{E1.1D}
\int_{\frac{|x|}{2}\leq |y|}\Big|\delta(V,x,y)-\tilde{\delta}(x,y)\Big|\frac{k_{\tau}(x,y)}{|y|^{d+2s}}\,dy\leq  &\int_{\frac{|x|}{2}\leq |y|\leq \frac{3|x|}{2}}|\delta(V,x,y)-\tilde{\delta}(x,y)|\frac{k_{\tau}(x,y)}{|y|^{d+2s}}\,dy\nonumber
\\
&\lesssim |x|^{\upgamma-2s-d}  \int_{\frac{|x|}{2}\leq |y|\leq \frac{3|x|}{2}} \,dy \leq C|x|^{\upgamma-2s}.
\end{align}
Choose $\theta\geq 0, \upgamma\in (s+\frac{1}{2}, 2s)$ satisfying 
$$\theta+\gamma-1>0,\quad \theta<(2s-\upgamma)(2s-1)
<(\upgamma-1)(2s-1).$$
Set $\upmu=\frac{\theta}{\upgamma(2s-1)}$.
Now suppose that $b_{\tau}(x)\cdot x\leq -|x|^{\theta+1}$ outside a fixed
compact set independent of $\tau$ and 
$$\sup_{\tau\in \mathcal{T}}\abs{b_\tau}\leq C(1+|x|^\theta),\quad \sup_{\tau\in \mathcal{T}}|g_{\tau}|\leq C(1+ |x|^{\frac{2s\theta}{2s-1}}) , $$
then 
\begin{equation}\label{E1.1E}
\limsup_{|x|\to\infty}\sup_{\tau \in\mathcal{T}}\frac{b_{\tau}(x)\cdot \nabla V(x)}
{|x|^{\theta+\upgamma-1}}<0.
\end{equation}
Thus, fixing $R=4$ in the above calculations, we get from 
\eqref{E1.1A},\eqref{E1.1B},\eqref{E1.1D} and \eqref{E1.1E} that
\begin{align*}
\sup_{\tau \in\mathcal{T}}\mathcal{L}_{\tau} [V](x)\leq k_0-k_1 |x|^{\theta+\gamma-1},
\end{align*}
for some suitable constants $k_0, k_1$.
Therefore to arrive at \eqref{Lyap}, we can take $h(x)=k_1 |x|^{\theta+\gamma-1}$. Moreover, since $\upgamma(1+\upmu)<2s$ we have
$V^{1+\upmu}\in L^1(\omega_s)$. By our choice of $\theta, \upmu$ it
is easily seen \eqref{EA.3A} holds and $V$ satisfies \eqref{EA.3B}.
\end{example}

%%%%%%%%%%%%%%%%%%%%%%%%%%%%%%%%%%%%%%%%%%%%%%%%%%%%%%%%%
\section{ \texorpdfstring{$\alpha$}{a}-discounted HJB Equation  with Rough Kernel}\label{S-alpha}

In this section we prove the existence of a unique classical solution
for the discounted problem. The results of this section will
be proved under a weaker setting compared to Assumption~\ref{assumptions}. In this section we consider operators of the form
\begin{equation}\label{E2.1}
\cI u(x)=\inf_{\tau \in\mathcal{T}}\bigl(\cL_{\tau} u + c_{\tau}(x) u(x)+ g_{\tau}(x)\bigr).
\end{equation}
Note that $c_{\tau}=-\alpha$ corresponds to the $\alpha$-discounted
problem. 
\begin{assumption}\label{A2.1}
We impose following assumptions on the coefficients of the equation \eqref{E2.1}.
\begin{itemize}
\item[(\hypertarget{B1}{B1})] For some positive constant $c_\circ$ we have
$$\sup_{\tau\in \mathcal{T}}c_{\tau}(x)\leq -c_\circ\quad x\in\rd.$$

\item[(\hypertarget{B2}{B2})] 
There exist an inf-compact $\sV\in C^2(\mathbb{R}^d), \sV\geq 0$ and a 
positive inf-compact function $ h\in C(\mathbb{R}^d)$, such that
\begin{align}\label{Lyap-alpha}
\sup_{\tau \in\mathcal{T}}
(\mathcal{L}_{\tau}\sV(x) + c_{\tau}(x) \sV(x)) \leq k_0 1_{\sK} - h(x),\quad x\in\rd,
\end{align}
for some $k_0>0$ and a compact set $\sK$. In addition, 
$\sup_{\tau \in\mathcal{T}}|g_{\tau}|\leq h$ and 
\begin{equation}\label{EA2.1A}
\lim_{|x|\to\infty}\,\frac{\sup_{\tau \in\mathcal{T}}|g_{\tau}(x)|}{h(x)}=0.
\end{equation}
\item[(\hypertarget{B3}{B3})] For some $\upmu\geq 0$ we have $\sV^{1+\upmu}\in L^1(\omega_s)$ and
\begin{align}
\sup_{\tau \in\mathcal{T}}\left[\frac{|b_\tau|}{(1+\sV)^{(2s-1)\upmu}}
+ \frac{|g_{\tau}|}{(1+\sV)^{1+2s\upmu}}
+ \frac{|c_{\tau}|}{(1+\sV)^{2s\upmu}}\right]&\leq C,\label{EA2.1B}
\\ 
\sup_{x, y}\frac{\sV(x+y)}{(1+\sV(x))(1+\sV(y))} + \limsup_{|x|\to\infty}\; \frac{1}{1+\sV(x)}\sup_{|y-x|\leq 1} \sV(y)& \leq C.\label{EA2.1C}
\end{align}
\item[(\hypertarget{B4}{B4})] The maps $x\mapsto k_{\tau}(x,y), b_\tau(x), g_{\tau}(x), c_{\tau}(x)$ are locally uniformly continuous
and locally bounded, uniformly in $\tau, y$.
\item[(\hypertarget{B5}{B5})] For every compact set 
$K\subset\rd$, there exists $\tilde\alpha\in (0, 1)$ so that
\begin{align*}
\sup_{\tau\in \mathcal{T}}\abs{b_\tau(x)-b_\tau(x')} 
+ \sup_{\tau\in \mathcal{T}}(\abs{g_{\tau}(x)-g_{\tau}(x')}+ \abs{c_{\tau}(x)-c_{\tau}(x')})
&+ \sup_{\tau\in \mathcal{T}}\,\sup_{y\in\rd}
|k_{\tau}(x, y)-k_{\tau}(x', y)|
\\
&\leq C_K|x-x'|^{\tilde\alpha}.
\end{align*}
\end{itemize}
\end{assumption}
Assumptions (\hyperlink{B1}{B1}) and (\hyperlink{B4}{B4})
are standard whereas (\hyperlink{B5}{B5}) is generally imposed to obtain
$C^{2s+\upkappa}$ regularity of the solutions. Let us now
cite a class of examples that satisfy 
\eqref{Lyap-alpha},\eqref{EA2.1B} and \eqref{EA2.1C}.
\begin{example}
We recall the notations from Example~\ref{Eg1.1}. Suppose
that 
\begin{equation}\label{Eg2.1A}
\sup_{\tau\in \mathcal{T}}(b_\tau\cdot x)_+\leq C_0 |x|^{\upsigma + 1}\quad \text{where}\; \upsigma\in [0, 1].
\end{equation}
Choose $\upgamma\in (0, 2s)$ so that for $\upsigma=1$ we have
$c_\circ > \upgamma C_0$ where $c_\circ$ is given by 
(\hyperlink{B1}{B1}). Let $0\leq \sV\in C^2(\mathbb{R}^d)$ be such that $\sV(x)=|x|^\upgamma$ for $|x|\geq 1$. Then the calculation in Example~\ref{Eg1.1} reveals that
\begin{equation}\label{Eg2.1B}
\sup_{\tau\in \mathcal{T}}\int_{\rd}\delta(V, x, y)\frac{k_{\tau}(x, y)}{|y|^{d+2s}}\leq C(1+|x|^{\upgamma-2s})\quad x\in B^c_4.
\end{equation}
Therefore, if we set $h(x)=\kappa |x|^{\upgamma}$ for $|x|\geq 1$, from
\eqref{Eg2.1A} and \eqref{Eg2.1B} it is easily seen that
\eqref{Lyap-alpha} holds for suitable $\kappa$ and compact set $\sK$. For (\hyperlink{B3}{B3}) to hold, we can choose
$\upmu\in [0, \frac{2s}{\upgamma}-1)$ and restrict 
the family $\{b_\tau\}_{\tau\in \mathcal{T}}$ further to satisfy
$$\sup_{\tau\in \mathcal{T}}|b_\tau(x)|\leq C(1+|x|)^{\upgamma\upmu(2s-1)}.$$
\end{example}
It is quite possible for $b_\tau$ to have super-linear growth as
we show in the example below.
\begin{example}
Let $\upsigma, \theta$ are positive numbers satisfying
$$\upsigma-1<\theta<4s^2, \quad \upsigma< 2s(2s-1).$$
Consider a family of $\{b_\tau\}, \{c_{\tau}\}$ satisfying
$$\sup_{\tau\in \mathcal{T}}|b_\tau(x)|\leq C(1+|x|^\upsigma),
\quad \sup_{\tau\in \mathcal{T}} c_{\tau}(x)\leq -c_\circ(1+|x|^\theta).$$
Then we can choose $\upgamma\in (0, 2s)$ and 
$\upmu\in (0, \frac{2s}{\upgamma}-1)$ so that 
$$\upsigma\leq \upgamma \upmu (2s-1),\quad
\text{and}\quad \theta\leq 2s \upgamma\upmu.$$
Now let $\sV$ to satisfy $\sV(x)=|x|^\upgamma$ for $|x|\geq 1$. It can be easily checked that \eqref{Lyap-alpha},
\eqref{EA2.1B} and \eqref{EA2.1C} hold for $h(x)\asymp |x|^{\theta+\upgamma}$.
\end{example}
By $C^{\eta+}_{\rm loc}(\rd)$ we denote the set of functions
that are  in $C^{\eta+\upkappa}(K)$ for every compact $K$ and for some $\upkappa>0$, possibly depending on $K$. More precisely, a function $\ell\in C^{\eta+}_{\rm loc}(\rd)$ if and only if for every compact set $K$, there
exists $\upkappa>0$ such that $\ell\in C^{\eta+\upkappa}(K)$.
%The main result of this section is the following.
%\begin{thm}\label{T2.1}
%Suppose that (\hyperlink{B1}{B1})--(\hyperlink{B4}{B4}) hold.
%Then there exists $w\in C(\rd)\cap\sorder(\sV)$ satisfying
%\begin{equation}\label{ET2.1A}
%\inf_{\tau \in\mathcal{T}} \Bigl(\cL_{\tau} w + c_{\tau} w + g_{\tau}\Bigr) =0\quad \text{in}\; \rd.
%\end{equation}
%In addition, if (\hyperlink{B5}{B5}) holds, then 
%$w\in C^{2s+}_{\rm loc}(\rd)$ and it is the unique solution
%in the class $\sorder(\sV)$.
%\end{thm}
The remaining part of this section is devoted to
the proof of Theorem~\ref{T1.2}. We first solve
\eqref{ET1.2A} in bounded domains (Theorem~\ref{T2.2} below)
. Then, using $\sV$ as a barrier function and the stability
estimates from \cite{Schwab_Silvestre}, we could find a 
subsequence of solutions, as the domains increase to $\rd$,
that converge to a solution of \eqref{ET1.2A}. This is done in Lemma~\ref{L2.3}. In Lemma~\ref{L2.6} we then show that this solution, obtained as a limit, is in the class
$C^{2s+}_{\rm loc}(\rd)\cap\sorder(\sV)$. Combining these results we then complete the proof of Theorem~\ref{T1.2}.

We begin with a comparison principle 
which will be used in several places.

\begin{lem}\label{L2.1}
Let $\Omega$ be a bounded domain. Suppose that
(\hyperlink{B1}{B1}) holds.
Let $u\in C(\rd)\cap L^1(\omega_s)$ be a viscosity solution to 
$$\inf_{\tau \in\mathcal{T}} \Bigl(\cL_{\tau} u + c_{\tau} u + g_{\tau}\Bigr) \geq 0\quad \text{in}\; \Omega,$$
and $v\in C(\rd)\cap L^1(\omega_s)$ is a viscosity solution to
$$\inf_{\tau \in\mathcal{T}} \Bigl(\cL_{\tau} v+ c_{\tau} v + g_{\tau}\Bigr)\leq 0\quad \text{in}\; \Omega.$$
Furthermore, assume that either $u\in C^{2}(\Omega)$ or 
$v\in C^{2}(\Omega)$.
Then, if $v\geq u$ in $\Omega^c$, we have $v\geq u$ in $\rd$.
\end{lem}

\begin{proof}
Without any loss of generality assume that $v\in C^{2}(\Omega)$ and
therefore, it is a classical supersolution.
Suppose, to the contrary, that 
$\sup_\Omega(u-v)>0$. Define
$$t_0=\inf\{t>0\; :\; v + t> u\;\; \text{in}\; \rd\}.$$
It is evident that $t_0\leq 2 \sup_\Omega(u-v)$. Again, since
$\sup_\Omega(u-v)>0$, we must have $t_0>0$. Let $\psi(x)=v(x)+ t_0$. From the definition we have $\psi\geq u$. Since $\psi-u\geq t_0$ in $\Omega^c$ and $\Omega$ is bounded, it follows that $\psi(x)=u(x)$ for some $x\in \Omega$. Therefore, $\psi$ is a
valid test function at $x$. Hence, from the definition of
viscosity subsolution, we must have
$$0\leq \inf_{\tau \in\mathcal{T}} \Bigl(\cL_{\tau} \psi(x) +
c_{\tau}(x)\psi(x) + g_{\tau}(x)\Bigr)
=\inf_{\tau \in\mathcal{T}} \Bigl(\cL_{\tau} v(x) + c_{\tau}(x) v(x) + g_{\tau}(x)\Bigr)-c_\circ t_0\leq -c_\circ t_0<0.$$
But this is contradiction. Hence we must have $v\geq u$ in $\rd$.
\end{proof}
%%%%%%%%%%%%%%%%%%%%%%%%%%%%%%%%%%%%%%%%%%%%%%%%%%%%%%%%%%%
Next result is a  $C^{1, \gamma}$ regularity estimate. This
is a special case of Theorem~\ref{c_1,gamma_regularity} which
we prove in Section~\ref{S-regu}.
\begin{thm}\label{T-reg}
Let $u$ be a viscosity solution to
$$\inf_{\tau \in\mathcal{T}}
\bigg[I_{\tau}[u](x)+b_{\tau}(x) \cdot \nabla u(x)+g_{\tau}(x) \bigg]=0 \quad \text{in}\; B_1.
$$
Let $\sup_{\tau}\norm{b_{\tau}}_{L^\infty(B_1)}\leq C_0$ and
$\gamma\in (0, 2s-1)$. Suppose that for some modulus of continuity $\varrho$ we have
\begin{align*}
  |k_{\tau}(x_1,y)-k_{\tau}(x_2,y)|\leq \varrho(|x_1-x_2|)\quad \forall x_1,x_2 \in B_1,\; \tau\in \mathcal{T}, y\in\rd.
  \end{align*}
Then we have
\begin{align*}
\norm{u}_{C^{1,\gamma}(B_{\frac{1}{2}})}\leq C\left(\norm{u}_{L^\infty(\mathbb{R}^d)}+\sup_{\tau\in\mathcal{T}}\norm{g_{\tau}}_{L^\infty(B_1)}\right),
 \end{align*}
 where the constant $C$ depends on $d,s,C_0,\varrho,\lambda, \Lambda$.
\end{thm}

Next we prove an existence result in the bounded domains.

\begin{thm}\label{T2.2}
Let $\Omega$ be a bounded $C^1$ domain. 
Assume that (\hyperlink{B1}{B1}) and (\hyperlink{B4}{B4}) hold.
Then  there exists a 
viscosity solution $W$ satisfying
\begin{equation}\label{ET2.2A}
\inf_{\tau \in\mathcal{T}} \Bigl(\cL_{\tau} W + c_{\tau} W  + g_{\tau}\Bigr)=0\quad \text{in}\; \Omega, \quad
\text{and}\quad W=0\quad \text{in}\; \Omega^c.
\end{equation}
In addition, if (\hyperlink{B5}{B5}) holds, then $W\in C^{2s+}_{\rm loc}(\Omega)\cap C^\upkappa(\rd)$ for some $\upkappa>0$ and $W$ is the unique
solution to \eqref{ET2.2A}.
\end{thm}

\begin{proof}
Existence of a viscosity solution follows from \cite[Corollary~5.7]{MOU-2017}. 

Next we show that $W\in C^\upkappa(\rd)$ for 
some $\upkappa>0$. Let $M=c_\circ^{-1}\sup_{a}\norm{g_{\tau}}_{L^\infty(\Omega)}$. Using Lemma~\ref{L2.1} we obtain $W\leq M$. Analogous argument also gives $W\geq -M$.
Therefore, using the barrier function of \cite[Lemma~5.10]{MOU-2017}, it is standard to show that $|W|\leq CM\delta^\upkappa(x)$, where $\delta(x)={\rm dist}(x, \Omega^c)$
(see for instance, \cite[Theorem~2.6]{B20}). Again, by
\cite[Theorem~7.2]{Schwab_Silvestre}, $W$ is H\"{o}lder
continuous in the interior. It is now quite standard to
show that $W\in C^\upkappa(\rd)$ for some $\upkappa>0$
(see again, the proof of \cite[Theorem~2.6]{B20}).  
From Theorem~\ref{T-reg} we also see that 
$W\in C^{1, \gamma}(\Omega_\sigma)$ for some $\gamma\in (0,2s-1)$ where 
$$\Omega_\sigma=\{x\in\Omega\; :\; {\rm dist}(x, \Omega^c)>\sigma\}, \quad \sigma>0.$$
 Therefore, by (\hyperlink{B5}{B5}), the function
$$x\mapsto b_\tau(x)\cdot \grad W(x) + c_{\tau}(x) W(x) + g_\tau(x)$$
is H\"{o}lder continuous in $\Omega_\sigma$, uniformly in
$\tau$, for every $\sigma>0$. Thus, using \cite[Theorem~1.3]{Serra2015},
 we obtain that $W\in C^{2s+\upkappa}(\Omega_\sigma)$ for some $\upkappa>0$. Thus $W\in C^{2s+}_{\rm loc}(\Omega)\cap C^\upkappa(\rd)$. In particular, $W$ is a classical solution to \eqref{ET2.2A}.

To establish the uniqueness, let $\tilde{W}$ be a viscosity solution to
\eqref{ET2.2A}. The preceding argument shows that $\tilde{W}$ is a classical solution. Hence, from the ellipticity property, it follows that
$$\sup_{\tau\in \mathcal{T}}(\cL_{\tau}(W-\tilde{W}) + c_{\tau}(W-\tilde{W}))\geq 0 \quad \text{in}\; \Omega.$$
Since $0$ is a supersolution to the above equation, applying
Lemma~\ref{L2.1} it follows that $W\leq \tilde{W}$ in $\rd$. From the symmetry we also have $\tilde{W}\leq W$,  giving us $W=\tilde{W}$. This completes the proof.
\end{proof}

\begin{rem}
In several places below we use the H\"{o}lder estimate from \cite[Theorem~7.2]{Schwab_Silvestre}.
Though the results of \cite{Schwab_Silvestre} are proved for nonlocal parabolic operators, we can still apply the results to our setting by
treating the solutions as stationary functions of the time variable.
\end{rem}

Our next step is to construct a solution for
\eqref{ET2.2A} in the whole space.
\begin{lem}\label{L2.3}
Suppose that (\hyperlink{B1}{B1}),
(\hyperlink{B2}{B2}) and (\hyperlink{B4}{B4}) hold.
Let $W_n$ be a viscosity solution to
\eqref{ET2.2A} in the ball $\Omega=B_n$. Then there exists
a subsequence $W_{n_k}$ such that $W_{n_k}\to w$, uniformly
on compacts and $w\in C(\rd)\cap L^1(\omega_s)$ satisfies
\begin{equation}\label{EL2.3A}
\inf_{\tau \in\mathcal{T}} \Bigl(\cL_{\tau} w + c_{\tau} w+ g_{\tau}\Bigr) =0\quad \text{in}\; \rd.
\end{equation}
Furthermore,  $|w(x)|\leq\frac{k_0}{c_\circ} + \sV(x)$ in $\rd$.
\end{lem}

\begin{proof}
Define $\tilde{\sV}=\frac{k_0}{c_\circ}+ \sV$ where $k_0$ is given by \eqref{Lyap-alpha}. It then follows from \eqref{Lyap-alpha} that
\begin{equation}\label{EL2.3B}
\sup_{\tau\in \mathcal{T}}\Bigl(\cL_{\tau} \tilde{\sV}+
c_{\tau} \tilde\sV + |g_\tau| \Bigr)\leq
\sup_{\tau\in \mathcal{T}} (\cL_{\tau} \sV + c_{\tau} \sV) + h
-k_0\leq 0\quad \text{in}\; \rd.
\end{equation}
Let $W_{n}$ be a viscosity solution to \eqref{ET2.2A} with
$\Omega=B_n$, that is,
\begin{equation}\label{EL2.3C}
\inf_{\tau \in\mathcal{T}} \Bigl(\cL_{\tau} W_n + c_{\tau} W_n + g_{\tau}\Bigr) =0\quad \text{in}\; B_n, \quad
\text{and}\quad W_n=0\quad \text{in}\; B_n^c.
\end{equation}
Applying Lemma~\ref{L2.1} on \eqref{EL2.3B} and \eqref{EL2.3C}, we see that 
$|W_{n}|\leq \tilde{\sV}$ in $\rd$, for all $n$. 

Next we show that for any compact set $\Kset$, $\{W_{n}\}_{n\geq n_0}$ is equicontinuous on $\Kset$, where $\Kset\Subset B_{n_0-4}$. Let $R^\prime$ be such that $\Kset\subset B_{R'}$. Without any loss of generality, we may assume that $R'+4 <n_0$.
Consider a smooth cut-off function $0\leq \chi \leq 1$ such that $\chi=1$ in the ball $B_{R^\prime+2}$  and $\chi=0$ outside the $B_{R^\prime +3}$. Let $\tilde{\psi}_{n}=\chi W_{n}$ in $\mathbb{R}^d$. From \eqref{EL2.3C} we see that
\begin{equation}\label{EL2.3D}
\inf_{\tau \in\mathcal{T}} \Bigl(\cL_{\tau}\tilde\psi_n + c_{\tau}\tilde\psi_n + g_{\tau} + I_{\tau}[(1-\chi)W_n]\Bigr) =0\quad \text{in}\; B_{R'+2},
\end{equation}
for all $n\geq n_0$. We claim that
\begin{equation}\label{EL2.3E}
\sup_{B_{R'+1}}\, \sup_{\tau\in \mathcal{T}} |I_{\tau}[(1-\chi)W_n]|
\leq C.
\end{equation}
Note that $|x+y|\geq R'+2$ and $|x|\leq R'+1$ imply that 
$|y|\geq 1$. Thus, for any $x\in B_{R'+1}$
\begin{align*}
\sup_{\tau\in \mathcal{T}} |I_{\tau}[(1-\chi)W_n]|
&\leq (2-2s)\int_{\rd} |\delta((1-\chi)W_n, x, y)|\frac{\Lambda}{|y|^{d+2s}} \dy
\\
&\leq 2 (2-2s) \int_{\abs{y}\geq 1} |W_n(x+y)| \frac{\Lambda}{|y|^{d+2s}} \dy
\\
&\leq 2 (2-2s) \int_{\abs{y}\geq 1} \tilde{\sV}(x+y) \frac{\Lambda}{|y|^{d+2s}} \dy
\\
&\leq C_{R'} \int_{\rd} \tilde{\sV}(x+y) \frac{\Lambda}{1+|x+y|^{d+2s}} \dy \leq C.
\end{align*}
This establishes \eqref{EL2.3E}. Also, 
$$
\sup_{\tau\in \mathcal{T}}\sup_{B_{R'+1}}|c_{\tau} W_n|\leq 
C_{R'}(1+ \sup_{B_{R'+1}}\sV).$$
Thus, applying \cite[Theorem~7.2]{Schwab_Silvestre} on \eqref{EL2.3D} , we obtain that for some $\alpha>0$
\begin{equation*}
\sup_{n\geq n_0}\, \norm{W_n}_{C^\alpha(B_{R'})}
= \sup_{n\geq n_0}\, \norm{\tilde\psi_n}_{C^\alpha(B_{R'})}\leq C.
\end{equation*}
Hence, $\{W_{n}\}_{n\geq n_0}$ is equicontinuous on $\Kset$. Applying Arzel\`{a}-Ascoli theorem and a standard diagonalization 
argument we have $W_{n_k}\to w\in C(\rd)$, uniformly on
compact sets, along some subsequence $n_k\to\infty$. Since
$|W_{n_k}|\leq \tilde{\sV}$ for all $n_k$, we also get
\begin{equation}\label{AB001}
|w|\leq \tilde{\sV}\quad
\text{and}\quad \lim_{n_k\to\infty} \norm{W_{n_k}-w}_{L^1(\omega_s)}=0.
\end{equation}
Finally, applying the stability property of viscosity solutions
\cite[Lemma~5]{CS11b}, we see that $w$ solves \eqref{EL2.3A}.
In fact, the stability property can be seen as follows: Suppose $\varphi\in C^2(B_{2\kappa}(x))$ touches  $w$ (strictly) from above at $x$ in $B_{2\kappa}(x)$, for some $\kappa>0$. From the 
local uniform convergence above,
we can find a sequence $(d_{n_k}, x_{n_k})\in \R\times B_{\kappa}(x)$ such that $x_{n_k}\to x, d_{n_k}\to 0$ and $\varphi+d_k$ would touch $W_{n_k}$ from above at the point $x_k$. Then letting
\[
v_{n_k}(y):=\left\{
\begin{array}{ll}
\varphi(y) + d_k & \text{for}\; y\in B_\kappa(x),
\\[2mm]
W_{n_k}(y) & \text{otherwise},
\end{array}
\right. 
\]
we obtain from \eqref{EL2.3C} and for large $n_k$ that
$$\inf_{\tau \in\mathcal{T}} \Bigl(\cL_{\tau} v_{n_k}(x_{n_k}) + c_{\tau} v_{n_k}(x_{n_k}) + g_{\tau}(x_{n_k})\Bigr) \geq 0.$$
Using Assumption~\ref{A2.1}(B4) and \eqref{AB001}, we can let $n_k\to\infty$ to get that
$$\inf_{\tau \in\mathcal{T}} \Bigl(\cL_{\tau} v(x) + c_{\tau} v(x) + g_{\tau}(x)\Bigr) \geq 0,$$
where $v$ is defined in the same fashion replacing $W_{n_k}$ by $w$ and $d_k$ by $0$. Similarly, we can also check that $w$ is a supersolution.
This completes the proof.
\end{proof}

Now we show that the solution $w$ obtained in Lemma~\ref{L2.3} belongs to the class $\sorder(\sV)\cap C^{2s+}_{\rm loc}(\rd)$.
\begin{lem}\label{L2.6}
Let $w$ be the solution to \eqref{EL2.3A} obtained in Lemma~\ref{L2.3}. It holds that $w\in\sorder(\sV)$. In addition, if (\hyperlink{B3}{B3})-(\hyperlink{B5}{B5}) holds, then any solution $\upsilon$ of \eqref{EL2.3A} satisfying 
$\upsilon\in\mathcal{O}(\sV)$
is in
$C^{2s+}_{\rm loc}(\rd)$. In particular, $w$ is a classical solution to \eqref{EL2.3A}.
\end{lem}

\begin{proof}
First we show that the solution $w$ obtained in Lemma~\ref{L2.3} is in $\sorder(\sV)$. Recall that $w=\lim_{n_k\to\infty} W_{n_k}$ where $W_{n_k}$ solves 
\eqref{ET2.2A} in the ball $B_{n_k}$. Fix $\varepsilon>0$.
By \eqref{EA2.1A} we have 
$$
(\varepsilon h(x)- \sup_{\tau \in\mathcal{T}}|g_{\tau}|(x))\geq 0
$$
for all $|x|$ large. 
Thus we can find a compact set 
$\Kset_\varepsilon\Supset \sK$ so that
\begin{equation}\label{ET2.4B}
\begin{split}
(\varepsilon h- \sup_{\tau \in\mathcal{T}}|g_{\tau}|(x)) & \geq 0 
\quad \text{in} \; \Kset^c_\varepsilon,
\\
\sup_{\tau\in \mathcal{T}}(\cL_{\tau} \varphi_\varepsilon
+ c_{\tau} \varphi_\varepsilon + |g_{\tau}|)
\leq \sup_{\tau\in \mathcal{T}}(\cL_{\tau} \varphi_\varepsilon 
+ c_{\tau} \varphi_\varepsilon + \varepsilon h) &\leq 0\quad \text{in}\; \Kset^c_\varepsilon,
\end{split}
\end{equation}
where $\varphi_\varepsilon=\varepsilon \sV$. Let 
$\kappa=\sup_{n_k}\,\max_{\Kset_\varepsilon}|W_{n_k}|<\infty$. From (\hyperlink{B1}{B1}) and
\eqref{ET2.4B} it follows that
$$\sup_{\tau\in \mathcal{T}}(\cL_{\tau} (\kappa+\varphi_\varepsilon) 
+ c_{\tau} (\kappa+\varphi_\varepsilon)+ |g_{\tau}|) \leq 0\quad \text{in}\; \Kset^c_\varepsilon.$$
Applying Lemma~\ref{L2.1} in the domain $B_{n_k}\setminus \Kset_\varepsilon$ we obtain
\begin{equation*}
|W_{n_k}|\leq \kappa + \varphi_\varepsilon= \kappa + \varepsilon \sV
\quad \text{in}\; \rd.
\end{equation*}
for all $n_k$ large. Hence, letting $n_k\to\infty$, we get
$|w|\leq \kappa + \varepsilon \sV$ in $\rd$ which in turn, implies
$$\limsup_{|x|\to\infty} \frac{1}{1+ \sV(x)} |w|\leq \varepsilon.$$
From the arbitrariness of $\varepsilon$ we have $w\in\sorder(\sV)$.

Now we prove the second part. Let $\upsilon$ solve
\begin{equation}\label{EL2.6A}
\inf_{\tau \in\mathcal{T}} \Bigl(\cL_{\tau} \upsilon + c_{\tau} \upsilon+ g_{\tau}\Bigr) =0\quad \text{in}\; \rd,
\end{equation}
and $\upsilon\in \mathcal{O}(\sV)$, that is, $|\upsilon|\leq C(1+\sV)$ in
$\rd$.
 Fix a point $x_0\in \rd$ and let
$r=[1+ \sV(x_0)]^{-\upmu}$,
where $\upmu$ is given by (\hyperlink{B3}{B3}). Define $v(x)= \upsilon(x_0+ rx)$. It then follows from \eqref{EL2.6A} that
\begin{equation}\label{ET2.4C}
\inf_{\tau \in\mathcal{T}}\left(\tilde I_{\tau}[v] + r^{2s-1} b_\tau(x_0+rx)\cdot \grad v + r^{2s} c_{\tau}(x_0+rx) v  + r^{2s} g_{\tau}(x_0+ r x) \right)
=0 \quad \text{in}\; \rd,
\end{equation}
where 
$$\tilde I_{\tau}[u](x)=\int_{\rd}\delta(u, x, y) \frac{k_{\tau}(x_0+rx, ry)}{|y|^{d+2s}}\dy.$$
Now consider a smooth cut-off function $\xi:\rd\to[0, 1]$ satisfying $\xi=1$ in $B_{\frac{3}{2}}$ and $\xi=0$ in $B^c_2$.
We can re-write \eqref{ET2.4C} as
\begin{equation}\label{ET2.4D}
\inf_{\tau \in\mathcal{T}}\left(\tilde I_{\tau}[\psi] + r^{2s-1} b_\tau(x_0+rx)\cdot \grad \psi
+ r^{2s} c_{\tau}(x_0+rx) \psi + r^{2s} g_{\tau}(x_0+ r x) + \tilde I_{\tau}[(1-\xi)v] \right)
=0 \quad \text{in}\; B_1,
\end{equation}
for $\psi=\xi v$. Since $|\upsilon|\leq C(1+\sV)$,
from \eqref{EA2.1B} we have
\begin{equation}\label{ET2.4E}
\sup_{\tau\in \mathcal{T}}\, \sup_{B_1}(r^{2s-1} |b_\tau(x_0+r\cdot)| + r^{2s} \frac{|g_{\tau}(x_0+r\cdot)|}{1+\sV(x_0)})
\leq C,\quad 
\sup_{\tau\in \mathcal{T}}\,\sup_{B_1}  r^{2s}\frac{|c_{\tau}(x_0+r\cdot)v|}{1+\sV(x_0)}\leq C,
\end{equation}
where $C$ is independent of $r$ and $x_0$.  From \eqref{EA2.1C} we also have $\norm{\psi}_{L^\infty(\rd)}\leq C (1+\sV(x_0))$. For $x\in B_1$ let us now
compute
\begin{align}\label{ET2.4F}
\sup_{\tau\in \mathcal{T}}|\tilde I_\tau[(1-\xi)v]|
&\leq 2(2-2s)\int_{\rd} |(1-\xi(x+y))v(x+y)|\frac{1}{|y|^{d+2s}}\dy
\nonumber
\\
&\leq 2(2-2s)\int_{|y|\geq 1/2} |v(x+y)|\frac{1}{|y|^{d+2s}}\dy
\nonumber
\\
&\leq C \int_{\rd} |\upsilon(x_0+ rx+ry)|\frac{1}{1+|y|^{d+2s}}\dy
\nonumber
\\
&\leq C \int_{\rd} (1 + \sV(x_0+ rx + ry))\frac{1}{1+|y|^{d+2s}}\dy
\nonumber
\\
&\leq C \int_{|y|\leq r^{-1}} \frac{(1 + \sV(x_0))}{1+|y|^{d+2s}}\dy
+ C \int_{|y|> r^{-1}} \frac{(1 + \sV(x_0+ rx + ry))}{1+|y|^{d+2s}}\dy
\nonumber
\\
&\leq C(1+\sV(x_0)) + C \int_{|z|> 1} (1 + \sV(x_0+ rx + z))
\frac{r^{2s}}{r^{d+2s}+|z|^{d+2s}}\dz
\nonumber
\\
&\leq  C(1+\sV(x_0)) + C (1+\sV(x_0))\int_{|z|> 1}\frac{1+\sV(z)}{1+|z|^{d+2s}}\dz
\nonumber
\\
&\leq C(1+\sV(x_0)),
\end{align}
where in the fifth and seventh line we use \eqref{EA2.1C}. Thus, using \eqref{ET2.4D},\eqref{ET2.4E},\eqref{ET2.4F} and Theorem~\ref{T-reg}, we obtain
$$\sup_{B_{\frac{1}{2}}}|\grad v|\leq C(1+\sV(x_0)).$$
Computing the derivative at $x=0$ gives us
\begin{equation}\label{ET2.4G}
|\grad \upsilon(x_0) |\leq C(1+ \sV(x_0))^{1+\upmu} , \quad x_0\in\rd.
\end{equation}
This gives an estimate on the growth of $|\grad \upsilon|$. Next we fix a point $x_0$ and let 
$\zeta(x)=\upsilon(x_0 +x)$. Let $\xi$ be the cut-off function
as chosen above. Then, letting $\varphi=\xi\zeta$, we obtain
\begin{equation}\label{ET2.4H}
\inf_{\tau \in\mathcal{T}}\left(\widehat I_{\tau}[\varphi] +  b_\tau(x_0+x)\cdot \grad \varphi
+ c_{\tau}(x_0+x)\varphi +  g_{\tau}(x_0+ x) + \widehat I_\tau[(1-\xi)\zeta] \right)
=0 \quad \text{in}\; \rd,
\end{equation}
where 
$$\widehat I_{\tau}[u](x)=\int_{\rd}\delta(u, x, y) \frac{k_{\tau}(x_0+x, ry)}{|y|^{d+2s}}\dy.$$
Let $x, x'\in B_1$. Using (\hyperlink{B5}{B5}) we get
\begin{align*}
&|\widehat I_{\tau}[(1-\xi)\zeta](x)-\widehat I_{\tau}[(1-\xi)\zeta](x')|
\\
&\leq \int_{|y|\geq \frac{1}{2}}|\delta((1-\xi)\zeta, x, y)|
\frac{|k_{\tau}(x_0+x, y)-k_{\tau}(x_0+x', y)|}{|y|^{d+2s}}\dy
\\
&\quad + 2(2-2s) \Lambda\int_{|y|\geq \frac{1}{2}}|(1-\xi(x+y))\zeta(x+y)-(1-\xi(x'+y))\zeta(x'+y)|
\frac{1}{|y|^{d+2s}}\dy
\\
&\leq  C |x-x'|^{\tilde\alpha} + C \int_{|y|\geq \frac{1}{2}}|(\xi(x'+y)-\xi(x+y))\zeta(x+y)|
\frac{1}{|y|^{d+2s}}\dy
\\
&\quad+ \int_{|y|\geq \frac{1}{2}}|\zeta(x+y)-\zeta(x'+y)|\frac{1}{|y|^{d+2s}}\dy
\\
&\leq C |x-x'|^{\tilde\alpha} + C |x-x'| (1+\sV(x_0)) +
C |x-x'|\int_{|y|\geq \frac{1}{2}}(1+\sV(x+y))^{1+\upmu}\frac{1}{|y|^{d+2s}}\dy
\\
&\leq C (1+\sV(x_0))^{1+\upmu}\max\{|x-x'|^{\tilde\alpha}, |x-x'|\}
\end{align*}
for all $\tau\in \mathcal{T}$, where in the third line we use \eqref{EA2.1C} and \eqref{ET2.4G}, and in the last line we use the fact $\sV^{1+\upmu}\in L^1(\omega_s)$. The above estimate gives H\"{o}lder continuity estimate of $\widehat I_{\tau}[(1-\xi)\zeta]$ uniformly in $\tau\in \mathcal{T}$. 
From \eqref{ET2.4H} and Theorem~\ref{T-reg} we first
observe that $\upsilon\in C^{1, \gamma}(B_1(x_0))$ for some 
$\gamma\in (0, 2s-1)$. Therefore, using above estimates, we can 
apply \cite[Theorem~1.3]{Serra2015} to obtain that 
$\upsilon\in C^{2s+\upkappa}(B_{1/2}(x_0))$ for some $\upkappa>0$. Hence $\upsilon\in C^{2s+}_{\rm loc}(\rd)$,
completing the proof.
\end{proof}

Now we can complete the proof of Theorem~\ref{T1.2}.

\begin{proof}[Proof of Theorem~\ref{T1.2}]
Existence and regularity follow from Lemma~\ref{L2.3}
and Lemma~\ref{L2.6}. So we only prove uniqueness.
 Let $\tilde{w}\in\sorder(\sV)$ be a viscosity solution to \eqref{EL2.3A}. From 
Lemma~\ref{L2.6} we see that 
$\tilde{w}\in C^{2s+}_{\rm loc}(\rd)$. Hence $w, \tilde{w}$ are classical solutions to \eqref{ET1.2A}. Now suppose, to the contrary, that $w\neq \tilde{w}$ and, without any loss of generality,
we let $(\tilde{w}-w)_+\gneq 0$. Letting 
$v=\tilde{w}-w$
we obtain from \eqref{ET1.2A} that
\begin{equation}\label{ET2.4L}
\sup_{\tau\in \mathcal{T}}(\cL_{\tau} v + c_{\tau} v)\geq 0\quad \text{in}\; \rd.
\end{equation}
If $\sup_{\rd} (\tilde{w}-w)_+>0$ is attained in $\rd$, say at
a point $x_0$, then we get from \eqref{ET2.4L} that
$0>-c_\circ v(x_0)\geq 0$ which is a contradiction. So we consider
the situation where 
\begin{equation}\label{EL2.4J}
\sup_{B_m} v_+\nearrow \sup_{\rd} v,
\end{equation}
as $m\to\infty$ and the sequence is strictly increasing. Define 
$\Kset=\sK\cup\{\sV\leq 1\}$ (see \eqref{Lyap-alpha}) and let 
$\kappa=\max_{\Kset}v_+$. Letting $\psi=v-\kappa$ we
have from \eqref{ET2.4L} that
\begin{equation}\label{ET2.4K}
\sup_{\tau\in \mathcal{T}}(\cL_{\tau} \psi + c_{\tau} \psi)\geq 0\quad \text{in}\; \rd.
\end{equation}
Using \eqref{EL2.4J}, we have $\psi\gneq 0$ in $\Kset^c$. Define
$$t_0=\inf\{t>0\; :\; t\sV>\psi\quad \text{in}\; \Kset^c\}.$$
It is evident that $t_0$ is finite, and since $\psi(z)>0$ for some
$z\in \Kset^c$, we also have $t_0>0$. We claim that
$t_0\sV$ must touch $\psi$ from above. Suppose that, $t_0\sV>\psi$ in
$\Kset^c$. Since $\psi\in\sorder(\sV)$, there exists a compact set
$\Kset_1\Supset\Kset$ such that $|\psi|<\frac{t_0}{2}\sV$ in $\Kset^c_1$. 
Again, on $\partial\Kset$, we have $\psi\leq 0
<t_0\leq t_0\sV$. Therefore
$(t_0\sV-\psi)>0$ in $\overline{\Kset_1\setminus\Kset}$. Thus we can find $\varepsilon>0$ small enough so that $(t_0-\varepsilon)\sV>\psi$
in $\Kset^c$ which will contradict the definition of $t_0$. Hence
$t_0\sV$ must touch $\psi$ at some point $x_0\in\Kset^c$. Since $\psi\leq 0$ in $\Kset$, we also have $t_0 \sV\geq \psi$ in $\rd$.
Applying $t_0\sV$ as a test function at $x_0$ in \eqref{ET2.4K} and using 
\eqref{Lyap-alpha} we get
$$0\leq t_0 \sup_{\tau\in \mathcal{T}} (\cL_{\tau}\sV(x_0) + c_{\tau}\sV(x_0))
\leq - t_0 h(x_0)<0,$$
which is a contradiction. Therefore $t_0$ must be zero and 
$(\tilde{w}-w)_+=0$. This gives us uniqueness, completing the proof.
\end{proof}

\begin{rem}\label{R2.1}
The proof of Theorem~\ref{T1.2} leads to the following Liouville type result. Suppose that there exist an inf-compact $C^2$ function $\sV\geq 0$ and a compact set $\sK$
satisfying
\begin{align*}
\sup_{\tau \in\mathcal{T}}
(\mathcal{L}_{\tau}\sV(x) + c_{\tau}(x) \sV(x)) < 0 \quad x\in\sK^c,
\end{align*}
and $c_\tau\leq -c_0<0$ for all $\tau\in\mathcal{T}$. Then, if $u\in\sorder(\sV)$ solves 
$$\inf_{\tau\in\mathcal{T}}(\mathcal{L}_{\tau} u(x) + c_{\tau}(x) u(x))=0\quad \text{in}\; \rd,$$
then $u\equiv 0$. The argument works for $s\in (0, 1)$ and  $u, b_\tau, c_\tau$ are only required to be continuous. This result should be compared with Meglioli \& Punzo \cite[Theorem~2.7]{MP22}. Unlike \cite{MP22},
the above equation is nonlinear and requires no additional regularity on drift. 
\end{rem}

\section{HJB for ergodic control problem}\label{Erg-HJB}
The main result of this section is the existence-uniqueness
result for ergodic HJB problem, that is, to find a
solution $(u, \lambda^*)$ satisfying
\begin{equation}\label{erg}
\inf_{\tau \in\mathcal{T}} \Bigl(\cL_{\tau} u + g_{\tau}\Bigr)-\lambda^*=0\quad \text{in}\; \rd.
\end{equation}
 Letting $c_{\tau}=-\alpha\in (-1, 0)$ in (\hyperlink{B1}{B1}), we see that $V$ in 
Assumption~\ref{assumptions} satisfies 
(\hyperlink{B2}{B2})--(\hyperlink{B3}{B3}). Thus Theorem~\ref{T1.2} is applicable under Assumption~\ref{assumptions}.
Let $w_\alpha$ be the unique
solution to 
\begin{equation}\label{E3.1}
\inf_{\tau \in\mathcal{T}} \Bigl(\cL_{\tau} w_\alpha + g_{\tau}\Bigr)-\alpha w_\alpha =0\quad \text{in}\; \rd,
\end{equation}
and $w_\alpha\in \sorder(V)\cap C^{2s+}_{\rm loc}(\rd)$. Define
$$\bar{w}_\alpha(x)=w_\alpha(x)-w_\alpha(0).$$
We claim that for some compact ball $B$ we have
\begin{equation}\label{E3.2}
|\bar{w}_\alpha(x)|\leq \max_{B}|\bar{w}_\alpha| + V(x)
\quad x\in\rd.
\end{equation}
Since $w_\alpha\in\sorder(V)$ and $V$ is inf-compact, we have 
$V-w_\alpha$ inf-compact. Thus, there exists $x_\alpha\in \rd$ 
satisfying $(V-w_\alpha)(x_\alpha)=\min_{\rd}(V-w_\alpha)$ implying
$w_\alpha\leq V + w_\alpha(x_\alpha)-V(x_\alpha)$. From \eqref{E3.1} we then have
$$0\leq \inf_{\tau \in\mathcal{T}}(\cL_{\tau} V(x_\alpha) + g_{\tau}(x_\alpha))
-\alpha w(x_\alpha)\leq k_0 - 
(h(x_\alpha)-\inf_{\tau \in\mathcal{T}}|g_{\tau}|(x_\alpha))+\alpha (V(0)-w_\alpha(0) - V(x_\alpha)),$$
by \eqref{Lyap}. Since $V$ is non-negative and 
$\alpha |w_\alpha|\leq k_0 + \alpha V$ (by Lemma~\ref{L2.3}), the
above inequality gives
$$(h(x_\alpha)-\sup_{\tau\in \mathcal{T}}|g_{\tau}|(x_\alpha))
\leq 2k_0 + 2\alpha V(0).$$
Since $\sup_{\tau\in \mathcal{T}}|g_{\tau}|\in\sorder(h)$ by (\hyperlink{A2}{A2}), there
exists a compact ball $B$, independent of $\alpha\in (0,1)$, so
that $x_\alpha\in B$. Hence
$$\bar{w}_\alpha(x)\leq \bar{w}_\alpha(x)-V(x) + V(x)\leq
 \bar{w}_\alpha(x_\alpha)-V(x_\alpha) + V(x)
 \leq \max_{B} \bar{w}_\alpha + V(x).$$
Again, since $(V+w_\alpha)$ is inf-compact, an argument similar to
above gives that 
$$\bar{w}_\alpha(x)\geq \min_{B} \bar{w}_\alpha-V,$$
for some compact ball $B$. Combining the above estimates we get \eqref{E3.2}.

\begin{lem}\label{L3.1}
It holds that $\sup_{\alpha\in(0,1)}\, \max_{B}|\bar{w}_\alpha|<\infty$.
\end{lem}

\begin{proof}
Suppose, to the contrary, that 
$$\sup_{\alpha\in(0,1)}\, \max_{B}|\bar{w}_\alpha|=\infty.$$
Therefore, we can find a sequence $\alpha_n\to\alpha_0\in [0,1]$
such that 
$$\max_{B}|\bar{w}_{\alpha_n}|\to \infty\quad \text{as}\; n\to\infty.$$
Since $|w_\alpha|\leq \frac{k_0}{\alpha} + V$, we must have $\alpha_0=0$. Denote by $\rho_n=\max_{B}|\bar{w}_{\alpha_n}|$.
From \eqref{E3.2} we observe that
 \begin{align}\label{EL3.1A}
 |\bar{w}_{\alpha_n}(x)|\leq  \rho_n + V(x)  \quad\text{in}\; \rd.
 \end{align}
Letting $\psi_n=\frac{1}{\rho_n} \bar{w}_{\alpha_n}$, we obtain
from \eqref{E3.1} that
 \begin{align}\label{EL3.1B}
 \inf_{\tau \in\mathcal{T}}[\mathcal{L}_{\tau} {\psi}_n + \rho^{-1}_n
 g_{\tau}]-\alpha_n \psi_n -\frac{\alpha_n}{\rho_n} w_{\alpha_n}(0)=0
 \quad \text{in}\; \rd.
 \end{align}
Following the arguments of Lemma~\ref{L2.3} and using \eqref{EL3.1A}, it can be easily seen that
for every compact $\Kset$, $\{\psi_n\}_{n\geq 1}$
is in $C^{\upkappa}(\Kset)$ for some $\upkappa>0$, uniformly in $n$. Hence there exists a $\psi\in C(\rd)$ such that
$$\psi_n\to \psi\quad \text{uniformly on compacts},$$
along some subsequence. From \eqref{EL3.1A} we also have
$|\psi|\leq 1$ and $\max_{B}|\psi|=1$. Using the stability property
of viscosity solution in \eqref{EL3.1B}, we obtain
\begin{equation}\label{EL3.1C}
\inf_{\tau \in\mathcal{T}}\cL_{\tau}\psi=0\quad \text{in}\; \rd.
\end{equation}
The argument of Lemma~\ref{L2.6} gives us 
$\psi\in C^{2s+\upkappa}_{\rm loc}(\rd)$. Hence $\psi$ is a classical solution to \eqref{EL3.1C}.
Since $|\psi|$ attends its maximum in $B$, there exists $z\in B$
satisfying $|\psi(z)|=1$. Suppose that $\psi(z)=1$. In view of
\eqref{EL3.1C}, this implies $\psi\equiv 1$ which contradicts 
the fact $\psi(0)=0$. We arrive at a similar contradiction when
$\psi(z)=-1$. This completes the proof.
\end{proof}

Next we prove the existence of a solution to the ergodic control problem.

\begin{thm}\label{T3.2}
Suppose that (\hyperlink{A1}{A1})--(\hyperlink{A4}{A4}) hold. Then for some 
sequence $\alpha_n\to 0$ we have
$$\lim_{n\to\infty} \bar{w}_{\alpha_n}=u,\quad 
\alpha_n w_{\alpha_n}(0)=\lambda^*$$
for some $(u, \lambda^*)\in C(\rd)\times\mathbb{R}$, and 
\begin{equation}\label{ET3.2A}
\inf_{\tau \in\mathcal{T}}(\cL_{\tau} u + g_{\tau}) -\lambda^*=0\quad \text{in}\; \rd.
\end{equation}
Moreover, if (\hyperlink{A5}{A5}) holds, then
$u\in \sorder(V)\cap C^{2s+}_{\rm loc}(\rd)$.
\end{thm}

\begin{proof}
From \eqref{E3.1} we note that
\begin{equation}\label{ET3.2B}
\inf_{\tau \in\mathcal{T}} \Bigl(\cL_{\tau} \bar{w}_\alpha + g_{\tau}\Bigr)-\alpha \bar{w}_\alpha - \alpha w_\alpha(0)=0\quad \text{in}\; \rd.
\end{equation}
From Lemma~\ref{L3.1}, \eqref{E3.2} and the proof of 
Lemma~\ref{L2.3}, we see that the family $\{\bar{w}_\alpha\, :\, \alpha\in (0, 1)\}$ is locally equicontinuous. Also note that
$\alpha|w_\alpha(0)|\leq \kappa + \alpha V(0)$. Hence we can find
a sequence $\alpha_n\to 0$ such that 
\begin{align*}
\lim_{n\to\infty}\alpha_n w_{\alpha_n}(0)&=\lambda^*
\\
\lim_{n\to\infty}\bar{w}_{\alpha_n} & = u \quad \text{uniformly over compacts}.
\end{align*}
Using stability property of the viscosity solutions and passing the
limit in \eqref{ET3.2B} we obtain
\begin{equation}\label{ET3.2C}
\inf_{\tau \in\mathcal{T}}(\cL_{\tau} u + g_{\tau}) -\lambda^*=0\quad \text{in}\; \rd.
\end{equation}
This gives us \eqref{ET3.2A}. From \eqref{E3.2} we also have
$|u|\leq C+ V$ in $\rd$. Thus we can follow the proof of 
Lemma~\ref{L2.6} to conclude that $u\in 
C^{2s+}_{\rm loc}(\rd)$.

Thus we remain to show that $u\in\sorder(V)$. Fix an $\varepsilon>0$. Recall from Lemma~\ref{L2.6} that 
$w_{\alpha_n}\in\sorder(V)$. Thus, the
functions $\varepsilon V-w_{\alpha}$ and 
$\varepsilon V+w_{\alpha}$ are inf-compact. Since 
$\sup_{\tau\in \mathcal{T}}|g_{\tau}|\in\sorder(h)$, we
also have 
$\varepsilon h - \sup_{\tau\in \mathcal{T}}|g_{\tau}|$ inf-compact.
Thus the proof of \eqref{E3.2} works with $V$ replaced by $\varepsilon V$ and we obtain a compact $B(\varepsilon)$ satisfying
$$|\bar{w}_{\alpha_n}(x)|\leq \max_{B(\varepsilon)}|\bar{w}_{\alpha_n}| + \varepsilon V(x) \quad x\in\rd.$$
Letting $\alpha_n\to 0$ in the above equation we obtain
$$|u(x)|\leq \max_{B(\varepsilon)}|u| + \varepsilon V(x) \quad x\in\rd,$$
which in turn, gives
$$\limsup_{|x|\to\infty} \frac{|u(x)|}{1+ V(x)}\leq \varepsilon.$$
From the arbitrariness of $\varepsilon$ it follows that
$u\in\sorder(V)$, completing the proof.
\end{proof}

Now we complete the proof of Theorem~\ref{T1.1}.
%\begin{thm}\label{T3.3}
%Suppose that (\hyperlink{A1}{A1})--(\hyperlink{A5}{A5}) hold. Then there exists a
%unique pair $(u, \lambda^*)\in\sorder(V)\times\mathbb{R}$ satisfying
%\begin{equation}\label{ET3.3A}
%\inf_{\tau \in\mathcal{T}}(\cL_{\tau} u + g_{\tau}) -\lambda^*=0\quad \text{in}\; \rd, \quad u(0)=0.
%\end{equation}
%\end{thm}

 \begin{proof}[Proof of Theorem~\ref{T1.1}]
 Existence of $(u, \lambda)$ follows from Theorem~\ref{T3.2}.
So we only prove uniqueness. Suppose that 
$(\hat{u}, \hat{\rho})\in\sorder(V)\times\mathbb{R}$ is a solution
to \eqref{ET1.1A}. The proof of Lemma~\ref{L2.6} gives us that
$u, \hat{u}\in C^{2s+}_{\rm loc}(\rd)$, and therefore, 
both $u, \hat{u}$ are classical solutions.

Now suppose, to the contrary, that $\lambda^*\neq \hat{\rho}$.
Assume, without loss of generality, that $\hat{\rho}< \lambda^*$ and define $\hat{u}_\epsilon:= \hat{u}+\epsilon V$ for $\epsilon>0$. 
Since $\hat{u}$ and $V$ are classical solution to \eqref{ET1.1A} and
\eqref{Lyap}, respectively,
 we obtain
 \begin{align*}
&\inf_{\tau \in\mathcal{T}} [\mathcal{L}_{\tau} \hat{u}_\epsilon + g_{\tau}(x)]
\leq \inf_{\tau \in\mathcal{T}}(\mathcal{L}_{\tau} \hat{u} +g_{\tau}(x)) + \sup_{\tau\in \mathcal{T}} \mathcal{L}_{\tau}(\epsilon V)
\leq \hat{\rho} + \epsilon (k_0-h)\leq \hat{\rho} + \epsilon k_0.
 \end{align*}
We choose $\epsilon$ small enough such that $\hat{\rho} + \epsilon k_0 < \lambda^*$ which in turn, gives
\begin{equation}\label{ET3.3B}
\inf_{\tau \in\mathcal{T}} (\mathcal{L}_{\tau} \hat{u}_\epsilon + g_{\tau}(x))  < \lambda^*.
\end{equation}
Define $\phi_\epsilon= \epsilon V + \hat{u}-u$ and
since $u, \hat{u}\in\sorder(V)$, observe that $\phi_\epsilon(x)\longrightarrow +\infty $ as $|x|\longrightarrow +\infty$. In other words $\phi_\epsilon $ is inf-compact, and so it attains minimum at a point in $\rd$, say $y_\epsilon$.
Denote by $\kappa_\epsilon= (\epsilon V + \hat{u}-u)(y_\epsilon)$ and note that $u+\kappa_\epsilon\leq \hat{u}_\epsilon$. 
 From 
\eqref{ET1.1A} and \eqref{ET3.3B} we then obtain
 \begin{align*}
\lambda^*=\inf_{\tau \in\mathcal{T}}
 (\mathcal{L}_{\tau} u(y_\epsilon) + g_{\tau}(y_\epsilon)) =
\inf_{\tau \in\mathcal{T}}
 (\mathcal{L}_{\tau} (u+\kappa_\epsilon)(y_\epsilon) + g_{\tau}(y_\epsilon)) \leq \inf_{\tau \in\mathcal{T}}
 (\mathcal{L}_{\tau} \hat{u}_\epsilon (y_\epsilon) + g_{\tau}(y_\epsilon))<\lambda^*.
 \end{align*}
But this is a contradiction. Hence we must have $\hat{\rho}=\lambda^*$.

Next we show that $u=\hat{u}$. 
Denote by $v=\hat{u}-u$. Since $\hat{u}, u$ are classical solutions and $\hat{\rho}=\lambda^*$ we have
\begin{equation}\label{ET3.3C}
\sup_{\tau\in \mathcal{T}}\cL_{\tau} v \geq 0\quad \text{in}\; \rd.
\end{equation}
Let $\Kset=\{h\leq k_0\}\cup\{V\leq 1\}$. We claim that
\begin{equation}\label{ET3.3D}
\sup_{\rd} v = \max_{\Kset} v.
\end{equation}
Denote by $\hat{\kappa}=\max_{\Kset} v$ and $\hat{v}=v-\hat{\kappa}$. It is evident from \eqref{ET3.3C} that
$$\sup_{\tau\in \mathcal{T}}\cL_{\tau} \hat{v} \geq 0\quad \text{in}\; \rd.$$
Let
$$t_0=\inf\{t>0\; :\; \hat{v}< t V\quad \text{in}\; \Kset^c\}.$$
If $t_0=0$ we get the claim \eqref{ET3.3D}. For $t_0>0$,
the argument in Theorem~\ref{T1.2} gives that $t_0V\geq \hat{v}$ in $\rd$ and for some point $z\in\Kset^c$,
$t_0 V(z)=\hat{v}(z)$. Then 
$$0\leq \sup_{\tau\in \mathcal{T}}\cL_{\tau} \hat{v}(z)
\leq t_0 \sup_{\tau\in \mathcal{T}}\cL_{\tau} V(z)
\leq t_0 (k_0-h(z)) <0.$$
This is a contradiction. Hence we must have $t_0=0$, establishing \eqref{ET3.3D}. Let $\bar{z}\in\Kset$ be such that $\sup_{\rd} v= v(\bar{z})$. Computing \eqref{ET3.3C} at the point $z$ gives us
$$
0\leq\sup_{\tau\in \mathcal{T}}\cL_{\tau} v(\bar{z})
\leq \lambda \int_{\rd}\delta(v, \bar{z}, y) \frac{1}{|y|^{d+2s}}\dy\leq 0.
$$
Hence $v={\rm constant}$ which gives us $u=\hat{u}$ since
$u(0)=0=\hat{u}$. This completes the proof.
\end{proof}

%%%%%%%%%%%%%%%%%%%%%%%%%%%%%%%%%%%%%%%%%%%%%%%%%%%%%%%%%%%%%%%%%%
\section{HJB equation with mixed local-nonlocal operators}\label{S-mixed}

In this section we generalize the results of Section~\ref{Erg-HJB} to a more general class of integro-differential operator involving both local and nonlocal terms. More precisely, we consider the following operator
$$\cI u(x) =\inf_{\tau\in\mathcal{T}}\left[
{\rm tr}{(a_\tau(x)D^2u)} + \breve I_\tau [u](x) + b_\tau(x)\cdot\grad u + g_\tau \right],$$
where $\mathcal{T}$ is an indexing set, $a_\tau$ is positive
definite matrix and
$$ \breve{I}_{\tau} [u](x)=\int_{\rd} (u(x+y)-u(x)-
\grad u(x)\cdot y 1_{B_1}(y)) K_\tau(x, y) \dy.$$
\begin{assumption}\label{A4.1}
We impose the following standard assumptions on the coefficients.
\begin{itemize}
\item[(i)] $\{a_\tau\}_{\tau\in\mathcal{T}}$ is locally bounded, and 
$$ \lambda |\xi|^2\leq \xi a_\tau(x) \cdot\xi\leq \Lambda |\xi|^2\quad \text{for all}\; x, \xi\in\rd,$$
for some $0<\lambda\leq \Lambda$.
\item[(ii)] There exists a measurable $K:\rd\setminus\{0\}
\to (0, \infty)$, locally bounded in $\rd\setminus\{0\}$,
 such that
$$0\leq K_\tau(x, y)\leq K(y) \quad \text{for all}\; x\in\rd, 
y\in \rd\setminus\{0\},\, \tau\in\mathcal{T},$$
and
$$\int_{\rd} \min\{|y|^2, 1\} K(y)\, \dy<\infty.$$
\item[(iii)] For every compact set $\Kset$, there exists
$\tilde\alpha\in (0, 1)$ such that
$$\sup_{\tau\in\mathcal{T}}\norm{a_\tau(x)-a_\tau(x')}
+ \sup_{\tau\in\mathcal{T}}\abs{b_\tau(x)-b_\tau(x')}
+ \sup_{\tau\in\mathcal{T}}\abs{g_\tau(x)-g_\tau(x')}
\leq C_{\Kset}|x-x'|^{\tilde\alpha},$$
and
$$ \sup_{\tau\in \mathcal{T}}|K_\tau(x, y)-K_\tau(x', y)|\leq
C_{\Kset} |x-x'|^{\tilde\alpha} K(y)\quad y\in\rd\setminus\{0\},$$
for all $x, x'\in\Kset$.
\end{itemize}
\end{assumption}
Note that Assumption~\ref{A4.1}(i) is the non-degeneracy condition and Assumption~\ref{A4.1}(iii) is normally used 
to establish $C^{2, \upkappa}$ regularity. As before, we also
need Foster-Lyapunov type condition which we state below.
Denote by $\breve{K}(y)=1_{B^c_1}(y) K(y)$.

\begin{assumption}\label{A4.2}
The following hold.
\begin{itemize}
\item[(i)] There exist $V\in C^2(\mathbb{R}^d), V\geq 0$ and a function $0\leq h\in C(\mathbb{R}^d)$, $V$ and $h$ are inf-compact, such that
\begin{align}\label{Lyap-ellip}
\sup_{\tau \in\mathcal{T}} \left[
{\rm tr}{(a_\tau(x)D^2V)} + \breve I_\tau [V](x) + b_\tau(x)\cdot\grad V \right]
\leq k_0-h(x)
\quad x\in\rd,
\end{align}
for some $k_0>0$.
\item[(ii)] $\sup_{\tau\in \mathcal{T}}|g_{\tau}(x)|\leq h(x)$  and 
$\sup_{\tau \in\mathcal{T}}|g_{\tau}|\in\sorder(h)$.

\item[(iii)] For some $\upmu\geq 0$ we have $V^{1+\upmu}\in L^1(\breve{K})$ and
\begin{align}
\sup_{\tau \in\mathcal{T}}\,\frac{|g_{\tau}|}{(1+V)^{1+2\upmu}}&\leq C,\label{EA4.2A}
\\ 
\sup_{x, y}\frac{V(x+y)}{(1+V(x))(1+V(y))} &+ \limsup_{|x|\to\infty}\; \frac{1}{1+V(x)}\sup_{|y-x|\leq 1} V(y) \leq C.\label{EA4.2B}
\end{align}
Moreover, one of the following holds.
\begin{itemize}
\item[(a)] There exists a kernel $\mathscr{J}:\rd\setminus\{0\}\to (0, \infty)$, locally bounded in 
$\rd\setminus\{0\}$, such that
\begin{equation}\label{EA4.2C}
r^{d+2} K(ry)\leq \mathscr{J}(y) \quad \text{for all}\; r\in (0,1], \; y\in \rd\setminus\{0\},
\quad \int_{\rd} (|y|^2\wedge 1)\mathscr{J}(y)\dy<\infty, 
 \end{equation}
and
\begin{equation}\label{EA4.2D}
\sup_{\tau \in\mathcal{T}} \frac{|b_\tau|}{(1+V)^{\upmu}}\leq C,
\end{equation}
where $\upmu$ is same as above.
\item[(b)] $\sup_{\tau\in \mathcal{T}}\norm{b_\tau}_{L^\infty(\rd)}
<\infty$.
\end{itemize}
\end{itemize}
\end{assumption}
In the spirit of Theorem~\ref{T1.1}, We can complete the proof of Theorem~\ref{T1.3}
%\begin{thm}\label{T4.1}
%Suppose that Assumptions~\ref{A4.1} and ~\ref{A4.2}  hold. Then there exists a
%unique pair $(u, \lambda^*)\in\sorder(V)\times\mathbb{R}$ satisfying
%\begin{equation}\label{ET4.1A}
%\cI u(x) -\lambda^*=0\quad \text{in}\; \rd, \quad u(0)=0.
%\end{equation}
%\end{thm}

\begin{proof}[Proof of Theorem~\ref{T1.3}]
Since the proof is analogous to the proof of Theorem~\ref{T1.1},
we only provide the outline.
\medskip

\noindent\underline{$\alpha$-discounted problem:}
Applying \cite[Theorem~5.7]{Mou2019} we find a viscosity
solution $W_n$
\begin{equation*}%\label{ET4.1B}
\cI W_n -\alpha W_n=0\quad \text{in}\; B_n,
\quad \text{and}\quad W_n=0\quad \text{in}\; B^c_n,
\end{equation*}
where $\alpha\in (0,1)$.
Now by Assumption~\ref{A4.2} and the proof of Lemma~\ref{L2.3} we get $|W_n|\leq \frac{k_0}{\alpha} + V$.
Given $R'>0$ we choose $n_0$ large enough so that 
$R'+4\geq n_0$. Then using the notations of Lemma~\ref{L2.3}
we have
\begin{equation*}%\label{ET4.1C}
\inf_{\tau\in\mathcal{T}} 
\Bigl({\rm tr}{(a_\tau(x)D^2\tilde\psi_n)} + 
\breve{I}_\tau[\tilde\psi_n]+ b_\tau\cdot\tilde\psi_n + g_\tau + \breve{I}_\tau[(1-\chi)W_n]\Bigr) =0\quad \text{in}\; B_{R'+2},
\end{equation*}
where $\tilde\psi_n=\chi W_n$. Using Assumption~\ref{A4.1}(ii) and \eqref{EA4.2B} it can be easily checked that
$$\sup_{x\in B_{R'+1}}\sup_{\tau\in\mathcal{T}} 
|\breve{I}_\tau[(1-\chi)W_n]|<\infty.$$
Thus, by \cite[Theorem~4.2]{Mou2019},
$\{W_{n}\}_{n\geq n_0}$ is $\upkappa$-H\"{o}lder
continuous in $B_{R'}$,  uniformly in $n$. Therefore, we 
can repeat the proof of Lemma~\ref{L2.3} and find a 
subsequence $W_{n_k}\to w_\alpha\in \sorder(V)$ such that
\begin{equation}\label{ET4.1B}
\cI w_\alpha -\alpha w_\alpha=0\quad \text{in}\; \rd.
\end{equation}

\medskip

\noindent\underline{Unique solution of \eqref{ET4.1A}:}
As before, we denote by $\bar{w}_\alpha(x)=w_\alpha(x)-w_\alpha(0)$. The argument of Lemma~\ref{L3.1} goes through.
Therefore, we can repeat the proof of Theorem~\ref{T3.2}
to obtain that, along some subsequence $\alpha_n\to 0$, we
have $\bar{w}_{\alpha_n}\to u$, $\alpha_n w_{\alpha_n}(0)\to
\lambda^*$ and (passing limit in \eqref{ET4.1B})
\begin{equation}\label{ET4.1C}
\cI u -\lambda^*=0\quad \text{in}\; \rd, \quad u(0)=0.
\end{equation}
Proof of Theorem~\ref{T3.2} also gives us $u\in \sorder(V)$.
Next we argue that $u\in C^{2+}_{\rm loc}(\rd)$. To attain
this goal we need an estimate analogous to \eqref{ET2.4G}.
We use the notations of Lemma~\ref{L2.6}. Fix $x_0\in\rd$ and
let $r=[1+V(x_0)]^{-\upmu},\, v(x)=u(x_0 + rx)$ and
$\psi=\xi v$. Also, define
$$K^r_\tau= r^{d+2} K(x_0 + r x, ry),
\quad \text{and}\quad 
\breve{I}^r_\tau \xi(x)=\int_{\rd} (\xi(x+y)-\xi(x)-\grad\xi(x)\cdot y 1_{B_{\frac{1}{r}}}(y)) K^r_\tau (x, y)\dy.$$
It is easy to check from \eqref{ET4.1C} that
\begin{equation}\label{ET4.1D}
\inf_{\tau\in\mathcal{T}}\left[
{\rm tr}{(a_\tau(x_0+rx)D^2\psi)} + \breve{I}^r_\tau [\psi] + r b_\tau(x_0+rx)\cdot\grad \psi + r^2 g_\tau(x_0+rx)
+ \breve{I}^r_\tau[(1-\xi)v] \right]-\lambda^*=0
\quad \text{in}\; \rd.
\end{equation}
Recall that $|v(x)|\leq C(1+V(x_0+rx))$. Therefore, for
$x\in B_1$,
\begin{align*}
\sup_{\tau\in\mathcal{T}}|\breve{I}^r_\tau[(1-\xi)v]|
&\leq C r^{d+2}\int_{|y|\geq 1/2} (1 + V(x_0+rx+ry)) K(ry)\dy
\\
&= C r^2 \int_{|y|\geq \frac{1}{2r}} (1 + V(x_0+rx+y)) K(y)\dy
\\
&\leq C r^2 (1+V(x_0))\int_{\frac{1}{2r}\leq |y|\leq 1} K(y)\dy + C (1+V(x_0)) \int_{\rd} (1+V(y))\breve{K}(y)\dy
\\
&\leq C(1+V(x_0))\left[1 + \int_{|y|\leq 1} |y|^2 K(y) \dy
+ \norm{V}_{L^1(\breve{K})}
\right]
\end{align*}
for some constant $C$ independent of $x_0, r$, where
in the third line we use \eqref{EA4.2B}. Under Assumption~\ref{A4.2}(iii)(a), we can apply \cite[Theorem~4.2]{Mou2019} on \eqref{ET4.1D} to 
obtain $\eta\in (0,1)$ satisfying
\begin{equation}\label{ET4.1E}
|u(x_0)-u(x_0 + rx)|\leq C(1+V(x_0)) |x|^\eta\quad
\text{for all}\; |x|\leq \frac{1}{2},
\end{equation}
and for all $x_0\in \rd$.
For Assumption~\ref{A4.2}(iii)(b), we apply \cite[Lemma~2.1]{MZ21}, to obtain \eqref{ET4.1E}. Now consider $|x|\leq 1$. If $|x|\leq r$, then \eqref{ET4.1E} gives
$$|u(x_0)-u(x_0+x)|\leq C(1+V(x_0)) |x|^\eta.$$ 
Now suppose, $|x|> r$ and let $e=\frac{x}{|x|}$, 
$k=\lfloor \frac{2|x|}{r}+\frac{1}{2}\rfloor$. Note that
$$k\leq \frac{2|x|}{r}+\frac{1}{2}\leq k+1\; \Rightarrow 
\bigl|\frac{2|x|}{r}-k\bigr|
\leq \frac{1}{2}.$$
Using \eqref{EA4.2B} and \eqref{ET4.1E} we have
\begin{align*}
|u(x_0)-u(x_0+ x)|&\leq \sum_{i=1}^k|u(x_0 + \frac{r}{2}(i-1)e)
-u(x_0 + \frac{r}{2}ie) + |u(x_0 + \frac{r}{2}ke)-u(x_0+ x)|
\\
&\leq C(1+V(x_0)) \frac{k+1}{2^\eta}
\leq C_1(1+V(x_0)) \frac{|x|}{r}\leq C_1(1+V(x_0))^{1+\upmu}|x|^\eta.
\end{align*}
Combining the above estimate we get
$$|u(x_0)-u(x_0+ x)|\leq C(1+V(x_0))^{1+\upmu}|x|^\eta
\quad \text{for all}\; |x|\leq 1, \; \text{and}\; x_0\in\rd.
$$
Now using Assumption~\ref{A4.1}(iii), \cite[Theorem~5.3]{MZ21} and the proof of Lemma~\ref{L2.6} we see that 
$u\in C^{2+}_{\rm loc}(\rd)$. Hence $u$ is a classical solution to \eqref{ET4.1C}.

To complete the proof we only need to show uniqueness. Suppose that $(\hat{u}, \hat{\varrho})\in \sorder(V)\times\mathbb{R}$ is a solution to \eqref{ET4.1C}. From
the above argument we have $\hat{u}\in C^{2+}_{\rm loc}(\rd)$. Therefore, the arguments in Theorem~\ref{T1.1} gives that
$\hat\rho=\lambda^*$. The equality of $u=\hat{u}$ can be established using a similar argument, provided we could show 
that for $v=u-\hat{u}$, $\sup_{\rd} v= v(\bar{z})$ for
some $\bar{z}\in\rd$, implies $v=0$. We note from \eqref{ET4.1A} that
$$
\inf_{\tau\in\mathcal{T}}\left[
{\rm tr}{(a_\tau(x)D^2(-v))} + \breve I_\tau [-v](x) + b_\tau(x)\cdot\grad (-v)\right]\leq \cI \hat{u}-\cI u=0.$$
Letting $\hat{v}=v(\bar{z})-v$, we obtain from above that
$$\inf_{\tau\in\mathcal{T}}\left[
{\rm tr}{(a_\tau(x)D^2\hat{v})} + \breve I_\tau [\hat{v}](x) + b_\tau(x)\cdot\grad\hat{v}\right]\leq 0 \quad \text{in}\; \rd.$$
Thus $\hat{v}$ a non-negative super-solution. Fix any bounded domain $\Omega$. Choose $M$ large enough, so that $\hat{v}<M$
in $\Omega$. Defining $v_{M}=\hat{v}\wedge M$ we get from above that
$$\inf_{\tau\in\mathcal{T}}\left[
{\rm tr}{(a_\tau(x)D^2 v_M)} + \breve I_\tau [v_M](x) + b_\tau(x)\cdot\grad v_M\right]\leq 0 \quad \text{in}\; \Omega,$$
and $v_M(\bar{z})=0$. From the weak Harnack principle 
\cite[Theorem~3.12]{Mou2019}, it then follows that
$v_M=0$ in $\Omega$. Since $\Omega, M$ are arbitrary, we have
$v=0$ in $\rd$. This completes the proof.
\end{proof}

%%%%%%%%%%%%%%%%%%%%%%%%%%%%%%%%%%%%%%%%%%%%%%%%%%%%%%%%%%%%%%%%
\section{\texorpdfstring{$C^{1, \gamma}$}{c} regularity}\label{S-regu}
The main goal of this section is to establish $C^{1, \gamma}$
regularity of the viscosity solutions to the equation
 \begin{equation}\label{nonlinear_rough kernel}
\sA u(x):= \inf_{\tau \in\mathcal{T}} \sup_{\iota\in\mathfrak{I}}
\bigg[I_{\tau \iota}[u](x)+b_{\tau \iota}(x) \cdot \nabla u(x)+g_{\tau \iota}(x) \bigg]=0,
\end{equation}
where $\mathcal{T}, \mathfrak{I}$ are some index sets and
\begin{align*}
I_{\tau \iota}[u](x)=\int_{\mathbb{R}^d}\delta (u,x,y)\frac{k_{\tau \iota}(x,y)}{|y|^{d+2s}} \,dy.
\end{align*}
Note that \eqref{nonlinear_rough kernel} is more general than
the one used in Theorem~\ref{T-reg}.
For each $\tau \in \mathcal{T}, \iota \in\mathfrak{I}$, $k_{\tau \iota}$ is symmetric in $y$, that is, $k_{\tau \iota}(x, y)=k_{\tau \iota}(x, -y)$ and
$$(2-2s)\lambda \leq k_{\tau \iota}(x, y)\leq (2-2s)\Lambda
\quad \text{for all}\; x, y, \quad 0<\lambda\leq \Lambda.$$
By $\cL_0(s)$ we denote the class of all kernels $k$ satisfying the above relation.
  
Let us define the  Pucci extremal operators.
The maximal operators, with respect to the
class $\cL_0(s)$, are defined as follows.
 \begin{align*}
 M^+u(x)&=\sup_{I\in \mathcal{L}_0(s)} I u(x)
 = (2-2s)\int_{\rd} \frac{\Lambda \delta^+(u, x, y)- \lambda \delta^-(u, x, y)}{|y|^{d+2s}},
 \\
  M^-u(x)&=\inf_{I\in \mathcal{L}_0(s)} Iu(x)
  =(2-2s)\int_{\rd} \frac{\lambda \delta^+(u, x, y)- \Lambda \delta^-(u, x, y)}{|y|^{d+2s}}.
 \end{align*}
 
We impose the following assumptions on the coefficients.
  \begin{itemize}
  \item[(\hypertarget{H1}{H1})] $b_{\tau \iota}$, $g_{\tau \iota}$ are continuous and
  $$\sup_{\tau\in \mathcal{T},\iota \in\mathfrak{I}}\norm{b_{\tau \iota}}_{L^\infty}<\infty, \quad \text{and}\quad \sup_{\tau\in \mathcal{T},\iota \in \mathfrak{I}}\norm{g_{\tau \iota}}_{L^\infty}<\infty.$$
  \item[(\hypertarget{H2}{H2})] The map $x\rightarrow k_{\tau \iota}(x,y)$ is uniformly continuous, uniformly in $\tau, \iota$ and $y$, that is,
  \begin{align*}
  |k_{\tau \iota}(x_1,y)-k_{\tau \iota}(x_2,y)|\leq \varrho(|x_1-x_2|)\quad \forall x_1,x_2 \in \mathbb{R}^d,\; \tau\in \mathcal{T}, \iota \in\mathfrak{I}, y\in\rd,
  \end{align*}
  where $\varrho$ is the modulus of continuity. 
  \end{itemize}
As before, we denote by $\omega_s(y)$ the weight function $[1+|y|^{d+2s}]^{-1}$.
%{We  need a definition of a rescaled operator, that is 
%\begin{align*}
%&\tilde{I}^r_{\tau \iota}(z)[u](x)= r^{2s}\int_{\mathbb{R}^d}\delta(u,x,y)\frac{k_{\tau \iota}(z+rx,ry)}{|y|^{d+2s}}\,dy
%\end{align*} 
Let us define the following collection of test functions
which will be used to define a norm on $I$ : Let $M>0$, $\Omega$ be a fixed domain and $x \in \Omega$. Let us consider the following set of functions: 
\begin{align*}
\mathcal{D}^x_M:=\{\phi\in L^1(\mathbb{R}^d,\omega) &|\; \phi \in C^2(x),  x\in \Omega,\hspace{1mm} 
\norm{\phi}_{L^1(\omega)}\leq M, \quad \text{and}
\\
&\quad |\phi(y)-\phi(x)-(y-x)\cdot \nabla\phi(x)|\leq M|y-x|^2, \forall y \in B_1(x)\},
\end{align*}
where $\omega$ is a weight function and $\phi\in C^2(x)$ means that there is a quadratic  polynomial $q$ such that $\phi(y)=q(y) + \sorder(|x-y|^2)$
for $y$ sufficiently close to $x$.

Let $\Omega =  B_N$ be a ball of fixed radius $N$, now observe that if $\phi \in \mathcal{ D}^x_M$ and $ \omega_{s}(y)=(1+|y|^{d+2s})^{-1}$ is a particular weight function, then
\begin{align}
\label{bound_on_phi}
&\int_{B_1(x)}|\phi(y)-\phi(x)-(y-x)\cdot \nabla\phi(x)|\,dy\leq M\int_{B_1(x)}|x-y|^2\,dy\leq M|B_1| \notag
\\
\implies & \bigg|\int_{B_1(x)}\phi(y)\,dy-\int_{B_1(x)}\phi(x)\,dy-\int_{B_1(x)}(y-x)\cdot \nabla\phi(x)\,dy\bigg|\leq M|B_1| \notag
\\
\implies & \bigg|\int_{B_1(x)}\phi(y)\,dy-\phi(x)|B_1|\bigg|\leq M|B_1|\implies |\phi(x)||B_1| \leq  M|B_1|+\int_{B_1(x)}\frac{|\phi(y)|(1+|y|^{d+2s})}{1+|y|^{d+2s}}\,dy \notag
\\
\implies &|\phi(x)||B_1| \leq  M|B_1|+M(1+(1+N)^{d+2s})\implies |\phi(x)|\leq M C_{d, s, \Omega}.
\end{align}
It is noteworthy that the right hand side of the last inequality does not depend on any particular choice of  $\phi$ or $x$. In other words, we get the same bound as long as $x\in \Omega$ and $\phi \in \mathcal{D}^x_M$ for a fixed $M>0$. 
\begin{defi}
Let  $\Omega$ be a domain and $\omega $ be a weight function , then for any non local operator $I$, the norm $||I||_\omega$ with respect to the weight function $\omega$ is defined as follows.
\begin{align}\label{norm}
||I||_\omega=\sup_{x,M}\Big\{\frac{I[\phi](x)}{1+M}\; \big|\; \phi \in \mathcal{D}^x_M ,\; x \in \Omega, M>0\Big\}.
\end{align}
\end{defi}

Our main result of this section is the following $C^{1, \gamma}$ regularity estimate
of the solutions to \eqref{nonlinear_rough kernel}.
 \begin{thm}\label{c_1,gamma_regularity}
Let $s\in (1/2, 1)$. Also, 
\hyperlink{H1}{(H1)}-\hyperlink{H2}{(H2)} hold in $B_1$ and 
$$\sup_{\tau, \iota}\norm{b_{\tau \iota}}_{L^\infty(B_1)}\leq C_0.$$
Then there exists a $\gamma\in(0, 2s-1)$ such that for
any bounded viscosity solution $u\in C(\mathbb{R}^d)$  to \eqref{nonlinear_rough kernel} in $B_1$ we have
 \begin{align*}
\norm{u}_{C^{1,\gamma}(B_{\frac{1}{2}})}\leq C\left(\norm{u}_{L^\infty(\mathbb{R}^d)}+\sup_{\tau, \iota}\norm{g_{\tau \iota}}_{L^\infty(B_1)}\right),
 \end{align*}
 where the constant $C$ depends on $d,s,C_0,\varrho,\lambda, \Lambda$.
 \end{thm}
 We mainly follow  the ideas in \cite{Serra2015} and \cite{Serra_Parabolic} for the proof of Theorem~\ref{c_1,gamma_regularity}.
 \begin{lem} \label{Rescaled_Operator}
 Let $s>\frac{1}{2}$ and $I$ be an integro-differential operator, elliptic with respect to the  class $\mathcal{L}_0(s)$, in particular an operator of the form $I_{\tau\iota}$ in \eqref{nonlinear_rough kernel} . Given $x_0\in \mathbb{R}^d, r>0, c>0$, and $l(x)=a\cdot x+b$, define $\tilde{I}$ by 
 \begin{align*}
 \tilde{I}\bigg( \frac{w(x_0+r\cdot)-l(x_0+r\cdot)}{c}\bigg)(x)=\frac{r^{2s}}{c}I(cw)(x_0+rx)
 \end{align*}
 Then $\tilde{I}$ is elliptic with respect to $\mathcal{L}_0(s)$ with the same ellipticity constants.
 \end{lem}
 \begin{proof}
 To see this, let us introduce the notation $\tau_r(x)=x_0+r\cdot x$. Now from the  definition of  $\tilde{I}$ we observe that
 \begin{align*}
 &\tilde{I}((w-l)\circ \tau_r)(x)=\frac{r^{2s}}{c}I(cw)(\tau_r(x))
 \\
 \implies &\tilde{I}(w)(x)=\tilde{I}(w\circ \tau^{-1}_r\circ \tau_r)(x)=\frac{r^{2s}}{c}I(c(w+l)\circ \tau_r^{-1})(\tau_r(x))
 \end{align*}
 Now it is easy to see that 
 \begin{align*}
 \tilde{I}(u)(x)-\tilde{I}(v)(x)\leq &\frac{r^{2s}}{c}\Big[I(c(u+l)\circ \tau_r^{-1})(\tau_r(x))-I(c(v+l)\circ \tau_r^{-1})(\tau_r(x))\Big]
 \\
 \leq &\frac{r^{2s}}{c}M^+(c(u-v)\circ \tau^{-1}_r)(\tau_r(x))=M^+((u-v))(x),
 \end{align*}
 using the fact that the extremal operators $M^+$ and 
 $M^-$ are translation invariant. Hence the proof.
 \end{proof}

Proof of the following result can be found in \cite[Lemma~4.3]{Serra2015}.
 \begin{lem}\label{Least_square_{tau}pproximation} 
 Let $s> \frac{1}{2}, \beta \in (1,2s)$ and define the for any $z\in \mathbb{R}^d$ and $r>0$ the following affine function
 \begin{align*}
 l_{r,z}(x)=a^* \cdot (x-z)+b^*,
 \end{align*}
 where
 \begin{align*}
 a^*_i= \frac{\int_{B_r(z)}u(x)(x_i-z_i)\,dx}{\int_{B_r(z)}(x_i-z_i)^2\,dx} \quad \text{and} \quad b^*(r,z)= \fint_{B_r(z)} u(x)\,dx,
 \end{align*}
 equivalently,
 \begin{align*}
 (a^*,b^*)=\Argmin_{(a,b)\in\rd\times\R} \int_{B_r(z)} (u(x)-a(x-z)+b)^2 \,dx.
 \end{align*}
 If for some constant $C_0$ we have 
 \begin{align*}
 \sup_{r>0}\sup_{z\in B_{\frac{1}{2}}} r^{-\beta} ||u-l_{r,z}||_{L^\infty(B_r(z))}\leq C_0,
 \end{align*}
 then
 \begin{align*}
 \norm{u}_{C^\beta(B_{\frac{1}{2}})}\leq C(\norm{u}_{L^\infty(\mathbb{R}^d)}+C_0),
 \end{align*}
 where $C$ depends on the exponent $\beta $.
 \end{lem}
For our next result we need to introduce a class of
scaled operators. For $m\in \mathbb{N}$,
let $z_m\in B_{\frac{1}{2}}$ and  
\begin{equation}\label{E5.4}
\tilde{I}^m[w](x):= 
\inf_{\tau\in \mathcal{T}_m}\sup_{\iota \in\mathfrak{I}_m}\int_{\mathbb{R}^d} \delta(w,x,y) \frac{k_{\tau, \iota}(z_m + r_m x, r_m y)}{|y|^{d+2s}}\,dy,
\end{equation}
where $r_m>0$ and $k_{\tau, \iota}\in\cL_0(s)$ for all 
$\tau\in \mathcal{T}_m, \iota \in\mathfrak{I}_m$. Now we recall the weak convergence of operators from \cite{CS11b}. A sequence of
operators $I_m$ is said to be weakly convergent to $I$ 
(with respect to a weight function $\omega$), if for every
$\varepsilon>0$ small and test function $\phi$
a point $x_0\in\Omega$, where $\phi$ is a quadratic polynomial in $B_\varepsilon(x_0)$ and $\phi\in L^1(\omega)$,
we have
$$I_m[\phi](x)\to I[\phi](x)\quad \text{uniformly in}\; B_{\varepsilon/2}(x_0).$$
We need the following result on weak convergence.
 \begin{lem}\label{L5.4}
Let $z_m\to z_0$ and $r_m\to 0$ 
as $m\to\infty$. Let $\tilde{\mathcal{I}}^{m}$ be defined as above and the family $\{k_{\tau \iota}\; :\; \tau\in \mathcal{T}_m, \tau \in\SB_m\}$ satisfy \hyperlink{H2}{(H2)} with the same modulus of continuity $\varrho$.
Then, there exists a subsequence $\tilde{\mathcal{I}}^{m_k}$ converges weakly to a translation invariant operator $I_0$ with $I_0(0)=0$ where
$I_0$ is elliptic with respect to the class $\sL_0(s)$.
\end{lem}

\begin{proof}
  To prove the lemma, let us first define a translation invariant operator $\hat{I}^m$ as follows
 \begin{align*}
 \hat{I}^m(z_0)[\phi](x)= \inf_{\tau\in \mathcal{T}_m}
 \sup_{\iota \in\mathfrak{I}_m}\int_{\mathbb{R}^d} \delta(\phi,x,y) \frac{k_{\tau \iota}(z_0,r_my)}{|y|^{d+2s}}\,dy .
 \end{align*}
We claim that for every bounded domain $\Omega$ we have
\begin{equation}\label{EL5.4A}
||\tilde{I}^{m}-\hat{I}^m||_{\omega_{s}}\xrightarrow{m \rightarrow \infty} 0,
\end{equation}
 where $\norm{\cdot}_{\omega_s}$ is given by \eqref{norm}.  To prove the convergence, consider a test function $\phi\in \mathcal{D}^M_x$, defined on $\Omega$, and observe from
 \hyperlink{H2}{(H2)} that
 \begin{align*}
 |\tilde{I}^{m}[\phi](x)-\hat{I}^m[\phi](x)|
 & = \sup_{(\tau, \iota)\in\mathcal{T}_m\times\mathfrak{I}_m}|\tilde{I}_{\tau \iota}^{m}[\phi](x)-\hat{I}_{\tau \iota}^m[\phi](x)|
 \\
 =&2M\varrho(|z_m+r_mx-z_0|) \int_{B_1}\frac{|y|^2}{|y|^{d+2s}}  \,dy
 + \varrho(|z_m+r_mx-z_0|) \int_{B^c_1}\frac{|\delta(\phi,x,y)|}{|y|^{d+2s}}\,dy.
 \end{align*} 
Next we carefully calculate the second integral. Using \eqref{bound_on_phi} we see that 
 \begin{align*}
 &\bigg|\int_{B^c_1}\frac{\delta(\phi,x,y)}{|y|^{d+2s}}\,dy\bigg|\leq \int_{B^c_1}\frac{|\phi(x+y)+\phi(x-y)-2\phi(x)|}{|y|^{d+2s}}\,dy
 \\
 \leq &\int_{B^c_1}\frac{1+|y|^{d+2s}}{|y|^{d+2s}}(|\phi(x+y)|+|\phi(x-y)|)\omega_{s}(y)\,dy+ C(s, \Omega)\frac{M}{2s}
 \leq MC_{s, \Omega}.
 \end{align*}
Thus, it follows from \eqref{norm} that
 \begin{align*}
\norm{\tilde{I}^{m}-\hat{I}^m}_{\omega_s}
 \leq C_{s, \Omega}\, \sup_{x\in\Omega}\varrho(z_m+r_m x-z-0),
 \end{align*}
 which gives us \eqref{EL5.4A}. Again, by \cite[Theorem~42]{CS11b}, there exists a subsequence $\hat{I}^{m_k}$ converges weakly to some $I_0$. Combining with \eqref{EL5.4A} it is easily seen that 
$\tilde{I}^{m_k}$ converges weakly to $I_0$. It is also
evident that $I_0(0)=0$ and $I_0$ is elliptic with respect to
the class $\cL_0(s)$ (cf. \cite[Lemma~4.1]{Serra_Parabolic}). 
\end{proof}

Now we can complete the proof of Theorem~\ref{c_1,gamma_regularity} with the help of
Lemmas~\ref{Rescaled_Operator},~\ref{Least_square_{tau}pproximation} and \ref{L5.4}.
 \begin{proof}[Proof of Theorem~\ref{c_1,gamma_regularity}]
 We prove the theorem by the method of contradiction.
First we fix the choice of $\gamma$. Let $\upgamma_\circ$
be the H\"{o}lder exponent obtained with respect to the Pucci operators
$M^{\pm}$ in \cite[Theorem~12.1]{CS09} (see also 
\cite[Theorem~2.1]{Serra_Parabolic}). Fix $\gamma\in (0, \min\{2s-1,\upgamma_\circ\})$.

Now  suppose that there exist $I_k,u_k, b^k_{\tau \iota}$ and $\{g^k_{\tau \iota}\}_k$ satisfying 
\begin{align*}
 & \sup_{\tau, \iota}|b^k_{\tau \iota}|\leq C_0 \quad \text{and} \quad 
||u_k||_{L^\infty(\mathbb{R}^d)} + \sup_{\tau, \iota}\norm{g^k_{\tau \iota}}_{L^\infty(B_1)} = 1,
\\
& \text{but}\quad ||u_k||_{C^\beta(B_{\frac{1}{2}})}\xrightarrow{k \rightarrow \infty} +\infty, \quad \text{where}\quad \beta=1+\gamma.
\end{align*}
In view of Lemma \ref{Least_square_{tau}pproximation}, there exist $a^*(k,r,z)$ and $b^*(k,r,z)$ such that
\begin{align}\label{E3.5}
\sup_k \sup_{r>0}\sup_{B_{\frac{1}{2}}}r^{-\beta} ||u_k -l_{k,r,z}||_{L^\infty(B_r(z))}= +\infty,
\end{align}
where
\begin{align*}
&(a^*(k,r,z),b^*(k,r,z))=\Argmin_{(a,b)\in \mathbb{R}^d\times \mathbb{R}}\int_{B_r(z)}(u_k(x)-a\cdot(x-z)+b)^2\,dx
\\
&\text{and} \quad  l_{k,r,z}(x)=a^*(k,r,z)\cdot (x-z)+ b^*(k,r,z).
\end{align*}
Now for any $r>0$, define $\Theta$ as follows: 
\begin{align}
\label{Theta}
\Theta(r):=\sup_k \sup_{r^\prime \geq r}\sup_{z\in B_{\frac{1}{2}}} (r^\prime)^{-\beta} ||u_k -l_{k,r^\prime, z}||_{L^\infty(B_{r^\prime}(z))}.
\end{align}
Since $\norm{u_k}_{L^\infty(\rd)}$ is finite, we see that $\Theta (r)< \infty$ for all $r> 0$, and therefore, \eqref{Theta} is well-defined.
Furthermore, from \eqref{E3.5} we observe that for any  $M>0$, however large, there exists $\tilde{r}>0, \tilde{k}\in \mathbb{N}$ and $\tilde{z}\in B_\frac{1}{2}$ such that
\begin{align*}
(\tilde{r})^{-\beta} ||u_{\tilde{k}}-l_{\tilde{k},\tilde{r},\tilde{z}}||_{L^\infty(B_{\tilde{r}}(\tilde{z}))}> M.
\end{align*}
Therefore, since $\Theta(r)\geq \Theta(\tilde{r})>M$ for any $0<r \leq  \tilde{r}$, we get  that  
$\Theta(r)\uparrow \infty $ as $r \downarrow 0$. As $\Theta(\frac{1}{m})\uparrow +\infty$ with $m\uparrow +\infty$, there exits $r_m \geq \frac{1}{m}, k_m \in \mathbb{N}$ and $z_m \in B_{\frac{1}{2}}$ such that
\begin{align*}
\frac{1}{2}\Theta(r_m)\leq\frac{1}{2}\Theta(\frac{1}{m})< (r_m)^{-\beta} ||u_{k_m}-l_{k_m,r_m,z_m}||_{L^\infty(B_{r_m}(z_m))}.
\end{align*}
It is easily seen that $r_m$ converges to zero. 
With this, let us define a new function $v_m$ as follows
\begin{align}
v_m(x):= \bigg(\frac{u_{k_m}-l_{k_m,r_m,z_m}}{r_m^\beta \Theta(r_m)}\bigg)(z_m +r_mx).
\end{align}
It is easy to see from the condition of minimality that 
\begin{align}\label{E3.8}
\int_{B_1} v_m \,dx=0; \quad \int_{B_1} v_m x_i \,dx=0;\quad \text{and}\quad ||v_m||_{L^\infty(B_1)}\geq 1/2.
\end{align}
We next claim that
\begin{align}\label{AB00}
||v_m||_{L^\infty(B_R)} \leq CR^\beta , \quad \forall R\geq 1.
\end{align}
To prove the above growth bound, we first observe that for any $z\in B_{\frac{1}{2}}, k\in \mathbb{N}$ and $r^\prime \geq r>0$
\begin{align*}
&||u_k-l_{k,r^\prime,z}||_{L^\infty(B_{r^\prime}(z))}\leq (r^\prime)^\beta \Theta(r^\prime)\quad \text{and}\quad ||u_k-l_{k,2r^\prime,z}||_{L^\infty(B_{2r^\prime}(z))}\leq (2r^\prime)^\beta \Theta(2r^\prime)
\\
\implies &||l_{k, 2r^\prime , z}-l_{k, r^\prime, z}||_{L^\infty(B_{r^\prime}(z))}\leq  ||u_k-l_{k,r^\prime,z}||_{L^\infty(B_{r^\prime}(z))} +||u_k-l_{k,2r^\prime,z}||_{L^\infty(B_{r^\prime}(z))}
\\
&\hspace{11em} \leq (r^\prime)^\beta \Theta(r^\prime)+(2r^\prime)^\beta \Theta(2r^\prime).
\end{align*}
Now by the monotonicity property of $\Theta$, we conclude for any $r^\prime \geq r>0$
\begin{align*}
||l_{k, 2r^\prime , z}-l_{k, r^\prime, z}||_{L^\infty(B_r(z))}
%\leq &|l_{k, 2r^\prime , z}(\tilde{z})-u_k(\tilde{z})|+|u_k(\tilde{z})-l_{k, r^\prime, z}(\tilde{z})|
%\\
%\leq & ||l_{k, 2r^\prime , z}-u_k||_{L^\infty(B_{2r^\prime})}+||u_k-l_{k, r^\prime, z}||_{L^\infty(B_{r^\prime})}
%\\
%\leq  &(2r^\prime)^{\beta} (2r^\prime)^{-\beta}||l_{k, 2r^\prime , z}-u_k||_{L^\infty(B_{2r^\prime})}+(r^\prime)^{\beta} (r^\prime)^{-\beta}||u_k-l_{k, r^\prime, z}||_{L^\infty(B_{r^\prime})}
%\\
\leq  &(2r^\prime)^{\beta} \Theta(2 r^\prime)+(r^\prime)^{\beta} \Theta(r^\prime) \leq (2^\beta+1) (r^\prime)^\beta \Theta(r).
\end{align*}
This implies that
\begin{align*}
\sup_{k,z}|b^*(k,2r^\prime, z)-b(k,r^\prime,z)|\leq C \Theta(r) (r^\prime)^\beta \quad \text{and}\quad \sup_{k,z}|a^*(k,2r^\prime,z)-a^*(k,r^\prime,z)|\leq C \Theta(r)(r^\prime)^{\beta-1}
\end{align*}
Now if we take $R=2^N, N\geq 1$ and $r^\prime_j=2^{j}r\geq r, j=1,2,\ldots, N,$ in the above inequalities, then
\begin{align}
\label{R_Growth_on_{tau}^*}
\frac{|a^*(k,Rr,z)-a^*(k,r,z)|}{r^{\beta-1} \Theta(r)}\leq & \sum^N_{j=1}\frac{|a^*(k,r^\prime_j,z)-a^*(k,r^\prime_{j-1},z)|}{r^{\beta-1} \Theta(r)} \notag 
\\
=& \sum^N_{j=1}2^{(\beta-1)(j-1)}\frac{|a^*(k,r^\prime_j,z)-a^*(k,r^\prime_{j-1},z)|}{2^{(\beta-1)(j-1)}r^{\beta-1} \Theta(r)} \notag
\\
\leq & C \sum_{j=1}^N 2^{(j-1)(\beta-1)}\leq C \frac{(2^{\beta-1})^N-1}{2^{\beta-1}-1}  \leq \frac{C}{2^{\beta-1}-1}R^{\beta-1}.
\end{align}
Similarly, we can also prove that 
\begin{align}
\label{R_growth_on_b^*}
\frac{|b^*(k, Rr,z)-b^*(k,r,z)|}{r^\beta \Theta(r)}\leq CR^\beta.
\end{align}
Now, observe that
\begin{align*}
||v_m||_{L^\infty(B_R)}\leq \frac{||u_{k_m}-l_{k_m,r_m,z_m}||_{L^\infty(B_{Rr_m}(z_m))}}{r_m^\beta \Theta(r_m)}\leq & R^\beta \frac{\Theta(Rr_m)}{\Theta(r_m)}+ \frac{||l_{k_m,Rr_m,z_m}-l_{k_m,r_m,z_m}||_{L^\infty(B_{Rr_m}(z_m))}}{r_m^\beta \Theta(r_m)}.
\end{align*}
Moreover, for any $x\in B_{Rr_m}(z_m))$
\begin{align*}
|l_{k_m,Rr_m,z_m}-l_{k_m,r_m,z_m}| &\leq |a^*(k_m,Rr_m,z_m)-a^*(k_m,r_m,z_m)||x-z_m| 
\\
&\qquad  + |b^*(k_m,Rr_m,z_m)-b^*(k_m,r_m,z_m)|
\\
&\leq  Rr_m|a^*(k_m,Rr_m,z_m)-a^*(k_m,r_m,z_m)| 
\\
&\qquad + |b^*(k_m,Rr_m,z_m)-b^*(k_m,r_m,z_m)|.
\end{align*}
By using \eqref{R_Growth_on_{tau}^*} and \eqref{R_growth_on_b^*}
we have
\begin{align}
\label{R_growth_on_l}
\frac{\norm{l_{k_m,Rr_m,z_m}-l_{k_m,r_m,z_m}}_{L^\infty(B_{Rr_m}(z_m))}}{r_m^\beta \Theta(r_m)}&\leq R\frac{|a^*(k_m,Rr_m,z_m)-a^*(k_m,r_m,z_m)|}{r_m^{\beta-1}\Theta(r_m)}\notag
\\
&\qquad + \frac{|b^*(k_m,Rr_m,z_m)-b^*(k_m,r_m,z_m)|}{r_m^\beta\Theta(r_m)} \notag
\\
\leq &C(\beta)R^\beta + C(\beta)R^\beta.
\end{align}
Ultimately,  using the monotonicity property of $\Theta$ and \eqref{R_growth_on_l} , we obtain
\begin{align}\label{AB01}
||v_m||_{L^\infty(B_R)}\leq (1+2C(\beta ))R^\beta.
\end{align}
This gives us \eqref{AB00}.

Now, observe that, by Lemma \ref{Rescaled_Operator},  $v_m$ satisfies the following equation in the viscosity sense
\begin{equation}\label{AB02}
\begin{split}
M^- v_m - r_m^{2s-1}C_0|\nabla v_m|- r_m^{2s-\beta}\frac{|b^k_{\tau \iota}\cdot  a_m^*|+|g^k_{\tau \iota}|}{ \Theta(r_m)} \leq 0  \hspace{2mm} \text{in} \hspace{2mm}  B_R,
\\
M^+ v_m + r_m^{2s-1}C_0|\nabla v_m|+r_m^{2s-\beta}\frac{|b^k_{\tau \iota}\cdot  a_m^*|+|g^k_{\tau \iota}|}{ \Theta(r_m)} \geq 0 \hspace{2mm} \text{in} \hspace{2mm}  B_R,
\end{split}
\end{equation}
where $a^*_m:=a^*(k_m,r_m, z_m)$.  We first claim that 
\begin{align*}
|a^*_m|\leq C(1+\Theta(r_m)).
\end{align*}
To see this, choose $l_m \in \mathbb{N}$ such that $2^{-l_m}\leq r_m < 2^{-(l_m-1)}$ and observe
\begin{align*}
|a^*(k_m, r_m, z_m)-a^*(k_m, 2^{(l_m-1)} r_m, z_m)| &\leq \sum^{l_m}_{j=1}|a^*( k_m, 2^jr_m, z_m)-a^*(k_m, 2^{(j-1)}r_m, z_m)|
\\
&\leq  \sum^{l_m}_{j=1}(2^{(j-1)}r_m)^{\beta-1}\frac{|a^*( k_m, 2^jr_m, z_m)-a^*(k_m, 2^{(j-1)}r_m, z_m)|}{(2^{(j-1)}r_m)^{(\beta-1)}}
\\
&\leq \sum^{l_m}_{j=1}(2^{(j-1)}r_m)^{\beta-1}\Theta(2^{j-1}r_m) 
\leq \Theta(r_m)r^{(\beta-1)}_m \frac{2^{(\beta-1)l_m}-1}{2^{\beta-1}-1}
\\
&\leq \Theta(r_m) \frac{(r_m2^{l_m})^{(\beta-1)}}{2^{\beta-1}-1} \leq \frac{\Theta(r_m)2^{\beta-1}}{2^{\beta-1}-1}.
\end{align*}
Again, from the definition  and monotonicity of $\Theta$
\begin{align*}
|u_{k_m}(x)-l_{k_m, 2^{(l_m-1)} r_m, z_m}(x)|&\leq \Theta (2^{l_m-1} r_m)(2^{l_m-1} r_m)^\beta \leq 
\Theta({1}/{2})(2^{l_m-1} r_m)^\beta  \quad \forall x\in B_{2^{l_m-1}r_m}(z_m)
\\
\implies|u_{k_m}(z_m)-b^*(k_m, 2^{(l_m-1)}r_m,z_m)|&\leq  \Theta({1}/{2})\left(2^{l_m-1} r_m\right)^\beta, \hspace{2mm} [\text{by substituting }\hspace{1mm} x=z_m]
\end{align*}
which in turn, implies that  for small $r_m$
\begin{align*}
|a^*(k_m, 2^{(l_m-1)} r_m, z_m)|\leq C \Theta(\frac{1}{2}).
\end{align*}
So we have our desired result
\begin{align*}
|a^*(k_m, r_m, z_m)|\leq |a^*(k_m, r_m, z_m)-a^*(k_m, 2^{(l_m-1)} r_m, z_m)|+|a^*(k_m, 2^{(l_m-1)} r_m, z_m)|\leq C(\Theta(r_m)+1).
\end{align*}
Therefore, as $b_{\tau \iota}, g_{\tau \iota}$ are uniformly bounded, we have 
\begin{align}\label{E5.16}
r_m^{2s-\beta}\frac{|b^k_{\tau \iota}\cdot  a_m^*|+|g^k_{\tau \iota}|}{ \Theta(r_m)} \xrightarrow{m \rightarrow \infty } 0.
\end{align}
Applying \cite[Theorem~7.2]{Schwab_Silvestre} it then follows from 
\eqref{AB01}-\eqref{AB02}, that the family $\{v_m\,:\, m\geq 1\}$ is locally H\"{o}lder continuous, uniformly in $m$. By the Arzel\`{a}-Ascoli
theorem we can extract a convergent subsequence of $v_m$.
Let $v_m\to v\in C(\rd)$
along some subsequence, as $m\to\infty$. It is also evident from \eqref{AB01} that
$$\norm{v}_{L^\infty(B_R)}\leq C(1+R^\beta).$$
 We claim that there exists a translation invariant 
operator $I_0$, elliptic with respect to $\cL_0(s)$, such that
\begin{equation}\label{AB03}
I_0(v)=0\quad \text{in} \hspace{2mm} \mathbb{R}^d.
\end{equation}
Once we have established \eqref{AB03}, it then follows from \cite[Theorem~3.1]{Serra_Parabolic} that $v(x)=a\cdot x+ b$. Passing the limit in the
first two equations of \eqref{E3.8} gives $a=0, b=0$ but it contradicts the third estimate in \eqref{E3.8} that requires 
$\norm{v}_{L^\infty(B_{1})}\geq 1/2$. Hence we have a contradiction
 to \eqref{E3.5}.

Thus we remain to show \eqref{AB03}. Choosing a further subsequence, if required, we may assume that $z_m\to z_0$
and $v_m\to v$ uniformly on compacts. Recall the operator
$\tilde{I}^m$ from \eqref{E5.4}. Now observe that 
for any bounded domain $\Omega$ we have
\begin{equation}\label{AB04}
\begin{split}
\tilde{I}^m v_m - r_m^{2s-1}K_0|\nabla v_m|- r_m^{2s-\beta}\frac{|b^k_{\tau \iota}\cdot  a_m^*|+|g^k_{\tau \iota}|}{ \Theta(r_m)} \leq 0  \hspace{2mm} \text{in} \hspace{2mm}  \Omega,
\\
\tilde{I}^m + r_m^{2s-1}K_0|\nabla v_m|+r_m^{2s-\beta}\frac{|b^k_{\tau \iota}\cdot  a_m^*|+|g^k_{\tau \iota}|}{ \Theta(r_m)} \geq 0 \hspace{2mm} \text{in} \hspace{2mm}  \Omega.
\end{split}
\end{equation}
By Lemma~\ref{L5.4}, $\tilde{I}^{m_k}$ converges weakly to
a translation invariant operator $I_0$ which is elliptic
with respect to $\cL$. Thus \eqref{AB03} follows from 
\eqref{E5.16} and \eqref{AB04}.
\end{proof}

\subsection*{Acknowledgement}
We thank the reviewers for helpful comments.
This research of Anup Biswas was supported in part by a SwarnaJayanti fellowship DST/SJF/MSA-01/2019-20.
 
 %%%%%%%%%%%%%%%%%%%%%%%%%%%%%%%%%%%%%%%%%%%%%%%%%%%%%%%%%%%%%%%%%%%%%

 \bibliographystyle{plain}
\bibliography{ref.bib}

\begin{thebibliography}{10}

\bibitem{Apple}
David Applebaum.
\newblock {\em L\'{e}vy processes and stochastic calculus}, volume~93 of {\em
  Cambridge Studies in Advanced Mathematics}.
\newblock Cambridge University Press, Cambridge, 2004.

\bibitem{ABC16}
Ari Arapostathis, Anup Biswas, and Luis Caffarelli.
\newblock The {D}irichlet problem for stable-like operators and related
  probabilistic representations.
\newblock {\em Comm. Partial Differential Equations}, 41(9):1472--1511, 2016.

\bibitem{red-book}
Ari Arapostathis, Vivek~S. Borkar, and Mrinal~K. Ghosh.
\newblock {\em Ergodic control of diffusion processes}, volume 143 of {\em
  Encyclopedia of Mathematics and its Applications}.
\newblock Cambridge University Press, Cambridge, 2012.

\bibitem{ACGZ}
Ari Arapostathis, Luis Caffarelli, Guodong Pang, and Yi~Zheng.
\newblock Ergodic control of a class of jump diffusions with finite {L}\'{e}vy
  measures and rough kernels.
\newblock {\em SIAM J. Control Optim.}, 57(2):1516--1540, 2019.

\bibitem{BCI08}
G.~Barles, E.~Chasseigne, and C.~Imbert.
\newblock On the {D}irichlet problem for second-order elliptic
  integro-differential equations.
\newblock {\em Indiana Univ. Math. J.}, 57(1):213--246, 2008.

\bibitem{BCCI12}
Guy Barles, Emmanuel Chasseigne, Adina Ciomaga, and Cyril Imbert.
\newblock Lipschitz regularity of solutions for mixed integro-differential
  equations.
\newblock {\em J. Differential Equations}, 252(11):6012--6060, 2012.

\bibitem{BCCI14}
Guy Barles, Emmanuel Chasseigne, Adina Ciomaga, and Cyril Imbert.
\newblock Large time behavior of periodic viscosity solutions for uniformly
  parabolic integro-differential equations.
\newblock {\em Calc. Var. Partial Differential Equations}, 50(1-2):283--304,
  2014.

\bibitem{BCI11}
Guy Barles, Emmanuel Chasseigne, and Cyril Imbert.
\newblock H\"{o}lder continuity of solutions of second-order non-linear
  elliptic integro-differential equations.
\newblock {\em J. Eur. Math. Soc. (JEMS)}, 13(1):1--26, 2011.

\bibitem{BI08}
Guy Barles and Cyril Imbert.
\newblock Second-order elliptic integro-differential equations: viscosity
  solutions' theory revisited.
\newblock {\em Ann. Inst. H. Poincar\'{e} C Anal. Non Lin\'{e}aire},
  25(3):567--585, 2008.

\bibitem{B20}
Anup Biswas.
\newblock Principal eigenvalues of a class of nonlinear integro-differential
  operators.
\newblock {\em J. Differential Equations}, 268(9):5257--5282, 2020.

\bibitem{BJK10}
Imran~H. Biswas, Espen~R. Jakobsen, and Kenneth~H. Karlsen.
\newblock Viscosity solutions for a system of integro-{PDE}s and connections to
  optimal switching and control of jump-diffusion processes.
\newblock {\em Appl. Math. Optim.}, 62(1):47--80, 2010.

\bibitem{Bor89}
Vivek~S. Borkar.
\newblock {\em Optimal control of diffusion processes}, volume 203 of {\em
  Pitman Research Notes in Mathematics Series}.
\newblock Longman Scientific \& Technical, Harlow; copublished in the United
  States with John Wiley \& Sons, Inc., New York, 1989.

\bibitem{BC19}
Cristina Br\"{a}ndle and Emmanuel Chasseigne.
\newblock On unbounded solutions of ergodic problems for non-local
  {H}amilton-{J}acobi equations.
\newblock {\em Nonlinear Anal.}, 180:94--128, 2019.

\bibitem{CS09}
Luis Caffarelli and Luis Silvestre.
\newblock Regularity theory for fully nonlinear integro-differential equations.
\newblock {\em Comm. Pure Appl. Math.}, 62(5):597--638, 2009.

\bibitem{CS11b}
Luis Caffarelli and Luis Silvestre.
\newblock Regularity results for nonlocal equations by approximation.
\newblock {\em Arch. Ration. Mech. Anal.}, 200(1):59--88, 2011.

\bibitem{CLN19}
Emmanuel Chasseigne, Olivier Ley, and Thi~Tuyen Nguyen.
\newblock A priori {L}ipschitz estimates for solutions of local and nonlocal
  {H}amilton-{J}acobi equations with {O}rnstein-{U}hlenbeck operator.
\newblock {\em Rev. Mat. Iberoam.}, 35(5):1415--1449, 2019.

\bibitem{CKS12}
Zhen-Qing Chen, Panki Kim, and Renming Song.
\newblock Dirichlet heat kernel estimates for fractional {L}aplacian with
  gradient perturbation.
\newblock {\em Ann. Probab.}, 40(6):2483--2538, 2012.

\bibitem{CGT22}
Adina Ciomaga, Daria Ghilli, and Erwin Topp.
\newblock Periodic homogenization for weakly elliptic
  {H}amilton-{J}acobi-{B}ellman equations with critical fractional diffusion.
\newblock {\em Comm. Partial Differential Equations}, 47(1):1--38, 2022.

\bibitem{FS06}
Wendell~H. Fleming and H.~Mete Soner.
\newblock {\em Controlled {M}arkov processes and viscosity solutions},
  volume~25 of {\em Stochastic Modelling and Applied Probability}.
\newblock Springer, New York, second edition, 2006.

\bibitem{FIL06}
Yasuhiro Fujita, Hitoshi Ishii, and Paola Loreti.
\newblock Asymptotic solutions of viscous {H}amilton-{J}acobi equations with
  {O}rnstein-{U}hlenbeck operator.
\newblock {\em Comm. Partial Differential Equations}, 31(4-6):827--848, 2006.

\bibitem{GM92}
M.~G. Garroni and J.-L. Menaldi.
\newblock {\em Green functions for second order parabolic integro-differential
  problems}, volume 275 of {\em Pitman Research Notes in Mathematics Series}.
\newblock Longman Scientific \& Technical, Harlow; copublished in the United
  States with John Wiley \& Sons, Inc., New York, 1992.

\bibitem{JK06}
Espen~R. Jakobsen and Kenneth~H. Karlsen.
\newblock A ``maximum principle for semicontinuous functions'' applicable to
  integro-partial differential equations.
\newblock {\em NoDEA Nonlinear Differential Equations Appl.}, 13(2):137--165,
  2006.

\bibitem{Kom84}
Takashi Komatsu.
\newblock On the martingale problem for generators of stable processes with
  perturbations.
\newblock {\em Osaka J. Math.}, 21(1):113--132, 1984.

\bibitem{MP22}
Guilia Meglioli and Fabio Punzo.
\newblock Uniqueness for fractional parabolic and elliptic equations with
  drift.
\newblock {\em ArXiv}, Preprint, 2022.

\bibitem{MOU-2017}
Chenchen Mou.
\newblock Perron's method for nonlocal fully nonlinear equations.
\newblock {\em Anal. PDE}, 10(5):1227--1254, 2017.

\bibitem{Mou2019}
Chenchen Mou.
\newblock Existence of {$C^\alpha$} solutions to integro-{PDE}s.
\newblock {\em Calc. Var. Partial Differential Equations}, 58(4):Paper No. 143,
  28, 2019.

\bibitem{Mou-Swiech}
Chenchen Mou and Andrzej \'{S}wiech.
\newblock Uniqueness of viscosity solutions for a class of integro-differential
  equations.
\newblock {\em NoDEA Nonlinear Differential Equations Appl.}, 22(6):1851--1882,
  2015.

\bibitem{MZ21}
Chenchen Mou and Yuming~Paul Zhang.
\newblock Regularity theory for second order integro-{PDE}s.
\newblock {\em Potential Anal.}, 54(2):387--407, 2021.

\bibitem{Schwab_Silvestre}
Russell~W. Schwab and Luis Silvestre.
\newblock Regularity for parabolic integro-differential equations with very
  irregular kernels.
\newblock {\em Anal. PDE}, 9(3):727--772, 2016.

\bibitem{Serra2015}
Joaquim Serra.
\newblock {$C^{\sigma+\alpha}$} regularity for concave nonlocal fully nonlinear
  elliptic equations with rough kernels.
\newblock {\em Calc. Var. Partial Differential Equations}, 54(4):3571--3601,
  2015.

\bibitem{Serra_Parabolic}
Joaquim Serra.
\newblock Regularity for fully nonlinear nonlocal parabolic equations with
  rough kernels.
\newblock {\em Calc. Var. Partial Differential Equations}, 54(1):615--629,
  2015.

\end{thebibliography}
\end{document}